\documentclass[12pt]{amsart}

\usepackage{amsmath, amssymb, amscd, newlfont}

\newcommand{\bbR}{{\mathbb{R}}}
\newcommand{\bbZ}{{\mathbb{Z}}}

\newcommand{\diag}{{\mathrm{diag}}}

\newcommand{\supp}{{\mathrm{supp}}}

\newcommand{\Ind}{{\mathrm{Ind}}}
\newcommand{\I}{{\mathrm{I}}}

\newcommand{\Sp}{{\mathrm{Sp}}}

\newcommand{\wit}{\widetilde}
\newcommand{\triv}{{\mathbf{1}}}

\numberwithin{equation}{section}

\newtheorem{Lem}[equation]{Lemma}

\newtheorem{Thm}[equation] {Theorem}
\newtheorem{Cor}[equation]{Corollary}
\newtheorem{Rem}[equation]{Remark}
\newtheorem{Exm}[equation]{Example}

\title
[
Degenerate Eisenstein Series
]
{
On the Images and Poles of 
Degenerate Eisenstein Series  for $GL(n, \Bbb A_\mathbb Q)$ and 
$GL(n, \Bbb R)$
} 

\author{Marcela Hanzer and Goran Mui\'c }
\address{
Department of Mathematics, 
University of Zagreb,
Bijeni\v cka 30, 10000 Zagreb,
Croatia}
\email{hanmar@math.hr}
\email{gmuic@math.hr}

\begin{document}

\begin{abstract}
I this paper we determine poles in the right--half plane and images of degenerate Eisenstein
series for $GL_n(\Bbb A_\mathbb Q)$ induced from a character on a
maximal parabolic subgroup. We apply those results to determine poles of  
degenerate Eisenstein series for $GL_n(\Bbb R)$.
\end{abstract}

\subjclass[2000]{11F70, 22E55}
\keywords{degenerate Eisenstein series, intertwining operators, $L$--functions}

\maketitle
\section{Introduction}

Since the seminal works by Selberg \cite{Sel1} and \cite{Sel2},
Eisenstein series have played an important role 
in the theory of automorphic forms. For the application of the adelic Eisenstein series 
to the trace formula one can consult recent fundamental works of 
Arthur \cite{Arthur-1} and \cite{Arthur-2}.  In the classical settings of a Lie group \cite{Langl} the Eisenstein series 
are also important especially in the recent applications to the
problems studied by the analytic theory of automorphic forms \cite{gold}.
The degenerate Eisenstein series for classical groups were used by
Piatetski--Shapiro and Rallis \cite{PR} to construct 
the principal $L$--functions on 
classical groups. Since then they have been used in the theory of
automorphic forms. Inspired by that work, the 
authors study generalization of their construction for classical symplectic and
orthogonal groups \cite{Muic1}, \cite{Muic2}, 
\cite{Han1}, 
\cite{Han2}. They are used to prove the unitarity of various isolated
local representations and for the explicit 
construction of certain classes
of square integrable automorphic representations given by \cite{Arthur-2}.

In the present paper we study the degenerate Eisenstein series for
$GL_n$ both adelically and classically on the Lie group $GL_n(\mathbb R)$. 
The Eisenstein series of that sort were 
already appeared in \cite{DGar}, and in a more classical language in
\cite{gold}.  The square--integrable automorphic forms were known from the fundamental works of Jacquet \cite{J1} 
and M\oe glin and Waldspurger \cite{MW}. Our main motivation comes
from the other two sources. First, although there exists a theory of Eisenstein series
on real reductive groups \cite{Langl} for some time, the poles of such
Eisenstein series, although important for the application in number theory, 
are almost a complete mystery. In this paper we will use adelic
methods to settle this in part. Adelic methods are more suitable for
computation of the constant term of an Eisenstein series, and usage of
representation theory in determination of poles of the same. The
application of such results to Eisenstein series on $GL_n(\mathbb R)$
is not straightforward as we explain below.  The other motivation 
for studying of (adelic) degenerate Eisentein series comes from
possible application in the theory of cohomology of discrete subgroups \cite{JS}.

We explain our results. We denote by $\Bbb A=\Bbb A_\mathbb Q$  the
ring of adeles of $\Bbb Q$. 
We fix two integers $m, n \ge 1$ such that $m\le n$. 
Let $\chi, \mu: \Bbb Q^\times\backslash  \Bbb A^\times\longrightarrow  \Bbb C^\times$ be unitary
 gr\" ossen characters. 
We consider the usual induced representation:
$$
I(s)\overset{def}{=}\Ind_{P_{m, n}(\Bbb A)}^{GL_{m+n}(\Bbb A)}\left(|\det |^{s/2}\chi \otimes |\det |^{-s/2}\mu\right),
$$
where $P_{m, n}$ is the block upper--triangular parabolic subgroup of $GL_{m+n}$ having a Levi subgroup isomorphic to 
$GL_m\times GL_n$. Let $f_s\in I(s)$ be a holomorphic section, then for 
$Re(s) > \frac{m+n}{2}$, the degenerate Eisenstein series 
$$
E(f_s)(g)\overset{def}{=}\sum_{\gamma \in P_{m, n}(\Bbb Q)\backslash GL_{m+n}(\Bbb Q)}
f_s(\gamma \cdot g)
$$
 converges absolutely and uniformly. It continues 
to a function which is meromorphic in $s$. Outside of poles, it 
is an automorphic form.

The main results about adelic degenerate Eisenstein series are the
following (stated in Section \ref{main-res}).  The first result 
is the following (see Theorem \ref{main-res-1}):

\begin{Thm}\label{a-int-1} Let $s>0$ be a real number. Assume that $m\le n$. Then, the degenerate Eisenstein
  series (\ref{de-1}) is  holomorphic and non--zero for all $s\not \in
  \{\frac{m+n}{2}-\alpha; \ \  \alpha\in \Bbb Z,  \ \ 0\le \alpha<
  \frac{m+n}{2} \}$. Moreover, the map  $f_s\mapsto E(f_s)$ is an embedding in the space of automorphic forms.
\end{Thm}

This solves the problem of understanding the behavior of degenerate
Eisenstein series on a generic set. A more specific information is 
contained in the following two theorems (see Theorems \ref{main-res-2}
and \ref{main-res-4}):

\begin{Thm}\label{a-int-2} Assume that $m\le n$. Then, the degenerate Eisenstein
  series (\ref{de-1}) with $\mu=\chi$ ($=$ to the trivial character
  without loss of the generality), for $s \in
  \{\frac{m+n}{2}-\alpha; \ \  \alpha\in \Bbb Z,  \ \ 0\le \alpha<
  \frac{m+n}{2} \}$,  has at most
  simple pole, and $f_s\mapsto (s-s_0)E(f_s)|_{s=s_0}$  is the unique spherical irreducible component in 
$\Ind_{P_{m, n}(\Bbb A)}^{GL_{m+n}(\Bbb A)}\left(|\det |^{s/2}
  \otimes |\det |^{-s/2}\right)$. Moreover, the pole occurs precisely
  for $1\le \alpha \le m-1$.
\end{Thm}

\begin{Thm}\label{a-int-3} Assume that $m\le n$. Then, for $\mu\neq
  \chi$ and $s \in
  \{\frac{m+n}{2}-\alpha; \ \  \alpha\in \Bbb Z,  \ \ 0\le \alpha<
  \frac{m+n}{2} \}$,   the  map (\ref{de-100}) is an embedding.
The degenerate Eisenstein series (\ref{de-1}) is  holomorphic if
  $\chi_\infty\mu^{-1}_\infty$ is not a
sign character, or if  $\alpha\le m$.
\end{Thm}

The proofs of theorems use  the usual strategy of computing constant terms along the standard Borel subgroup 
via normalized intertwining operators \cite{Sh1}, \cite{Sh2}.  But
one feature of  working with Eisenstein series that do not have 
square--integrable residues is unpleasant fact that they have huge
constant terms.

For example, if in  Theorem \ref{a-int-2}
we select $m=n$ and $\alpha=0$, then as a residue we obtain a trivial
representation of $GL_{m+n}(\Bbb A)$ which 
has just one term in its constant term. 
But for $\alpha\ge 1$ the situation is completely different: except
the longest element of the Weyl group much 
more Weyl group elements give a
non--trivial contribution. With this problem, the two usual
subproblems which should be considered,  
the computation of poles and images of local intertwining operators, and
computation and cancellation of the poles of the sums of 
global normalizing factors, are considerable more difficult.  

Thanks to the sufficiently developed local theory of degenerate
principal series due to Howe--Lee \cite{HL} 
(for $GL_n(\mathbb R)$) and Zelevinsky \cite{Z}  (for $GL_n(\mathbb
Q_p)$), we are able to find the elegant treatment to the problem of determining the poles 
and images of local normalized intertwining operators which diverge
from the usual heavily combinatorial computations. A very difficult combinatorial analysis  is still
 required to solve the problem of determination and  cancellation of the
 poles of the sums of global normalizing factors in the proof of 
 Theorem \ref{a-int-2} (Sections \ref{ct3000}).  The proof of Theorem
 \ref{a-int-1} is given in Section \ref{ct}, 
Theorem \ref{a-int-2}  is proved in Sections \ref{ct3000} and
\ref{conclusion}, and Theorem \ref{a-int-2} is proved in Section \ref{aaaaaaaaaaaaaaaaaaaaa}.

In Section \ref{cong-subgrps}, we consider a family of congruence subgroups   $\Gamma_0(N)$, $N\ge 1$, 
consisting of all matrices in $GL_{m+n}(\Bbb Q)$ 
$$
\left(\begin{matrix}a& b\\ c & d\end{matrix}\right)
$$
such that $a \in M_{m\times m}(\Bbb Z)$,  $b\in M_{m\times n}(\Bbb
Z)$, $d \in M_{n\times n}(\Bbb Z)$, and $c\in N \cdot M_{n\times m}(\Bbb Z)$,
and having determinant equal to $\pm 1$. Of course, $\Gamma_0(N)$
depends also on $m$ and  $n$ but we suppress this from the notation.

Their adelic framework is explained in Section \ref{cong-subgrps}. 
We use it to calculate restriction of adelic degenerate Eisenstein series to  $GL_{m+n}(\Bbb R)$ (see 
Section \ref{real-e}) for certain very specific choices of $f_s\in
I(s)$. We obtain the following degenerate Eisenstein series on
$GL_{m+n}(\mathbb R)$ 
$$
E_\infty(s, f_\infty)(g_\infty)=\sum_{\gamma \in \Gamma_0(N)\cap P_{m, n}(\Bbb Q)\backslash 
\Gamma_0(N)}   f_\infty(\gamma \cdot g_\infty), 
$$
where $f_\infty \in 
\Ind_{P_{m, n}(\Bbb R)}^{GL_{m+n}(\Bbb R)}\left(|\det |^{s/2} \otimes
  |\det |^{-s/2}\right)$.

Based on above global results, in Section \ref{main-res} we prove the following
(see Theorem \ref{main-res-3}): 

\begin{Thm}\label{int-res-3}  Let $s>0$ be a real number. Assume that
  $m\le n$. Then, we have the following:
\begin{itemize}
\item[(i)] The degenerate Eisenstein
  series $E_\infty(s, f_\infty)$, is holomorphic and non--zero for $s \not \in
  \{\frac{m+n}{2}-\alpha; \ \  \alpha\in \Bbb Z,  \ \ 0\le \alpha<
  \frac{m+n}{2} \}$.

\item[(ii)]  Let $s \in
  \{\frac{m+n}{2}-\alpha; \ \  \alpha\in \Bbb Z,  \ \ m\le \alpha<
  \frac{m+n}{2} \}$. Then, $E_\infty(s, \cdot)$ is holomorphic, and 
$E_\infty(s, f_\infty)\neq 0$   if and only if $\cal N(\Lambda_{s, \infty} ,
  \wit{w_0})f_\infty \neq 0$.   

\item[(iii)]  Let $s \in  \{\frac{m+n}{2}-\alpha; \ \  \alpha\in \Bbb Z,  \ \ 0\le \alpha\le m-1 \}$. Then, 
$E_\infty(s, \cdot)$ has at most simple  pole.

\item[(iv)] In the settings of (iii), if 
  $1\le \alpha\le m-1$, then $E_\infty(s, f_\infty)$ has a
  simple pole if and only if $\cal N(\Lambda_{s, \infty} ,
  \wit{w_0})f_\infty \neq 0$. 
\end{itemize}
\end{Thm}

We remark that $\cal N(\Lambda_{s, \infty},  \wit{w_0})$ is the local
normalized intertwining operator attached to the longest element of the Weyl group of 
$GL_{m+n}$ modulo that of the Levi $GL_m\times GL_n$ of $P_{m, n}$. 
If $f_\infty$ is spherical, then $\cal N(\Lambda_{s, \infty},
\wit{w_0})f_\infty \neq 0$  holds by  the Gindikin--Karpelevi\v c formula (\ref{de-5000000000}).

As we said above, Theorem \ref{int-res-3} is a consequence of above
results for adelic degenerate Eisenstein series, but the application
of those theorems is not straightforward. It proceeds in two
steps. In the first step we develop a criterion for non--vanishing of
$E(f_s)$ (ignoring the poles at this moment) considering single
$f_s\in I(s)$ which is manageable for application (see Corollary
\ref{cor-main-res-2}). Then, we apply to those specific $f_s$ used to 
construct $E_\infty(s, f_\infty)$ in Section \ref{real-e}. This gives
an interesting problem of determination whether or not certain functions
belong to the kernel of $\cal N(\Lambda_{s, p},
\wit{w_0})$ for $p$--adic places. Godement--Jacquet theory of local
zeta integrals for the principal $L$--function \cite{J} gives an elegant
solution to the problem.

By a general result of Langlands \cite{Langl}, all Eisenstein series
are holomorphic for $s=0$.  The reader will easily see that the poles of Eisenstein series for
$s< 0$ behave unpredictable even for $GL_2$. 

\vskip .2in 

We would like to thank N. Grbac, R. Howe, M. Tadi\' c, and J. Schwermer for some useful
discussions. The final version of this paper is prepared while we were visitors of  the Erwin
Schr\" odinger Institute in Vienna as a part of the Erwin
Schr\" odinger Institute programme Research in Teams. We would
like to thank J. Schwermer and the Institute for their hospitality.

\section{Degenerate Eisenstein Series and Their Constant Terms}\label{de}

We denote by $\Bbb A$  the ring of adeles of $\Bbb Q$. Let $K_p$ be defined by $GL_n(\Bbb Z_p)$, where
$\Bbb Z_p$ is the ring of integers of $\Bbb Q_p$. Then $K_p$ is a maximal compact subgroup 
of $GL_n(\Bbb Q_p)$. The group  $K_\infty=O(n)$  is a maximal compact subgroup of $GL_n(\Bbb R)$. Then the product 
$$ K=K_\infty\times \prod_{p} K_p$$
is a maximal compact subgroup of $GL_n(\Bbb A)$.

Let $m, n \ge 1$. We denote by $T=T_{m+n}$, $U=U_{m+n}$,  and $B_{m+n}$
the group of diagonal, upper triangular unipotent matrices, and  upper
triangular matrices in  $GL_{m+n}$. We denote by $W=W_{m+n}$ the Weyl
group of $T$. 
The group $W$ is isomorphic to the group of permutations 
of $m+n$--letters. For us it will be important to choose the representatives appropriately: 
we follow the convention introduced by Shahidi \cite{Sh1}, \cite{Sh2}. For $w\in W$, we denote by 
$\widetilde{w}$ its representative. Let $\Gamma$ be the set of roots of $T$ in $GL_{m+n}$. 
The set of roots which determine $U$ we denote by $\Sigma^+$. We denote by 
$\Delta\subset \Sigma^+$ uniquely determined set of simple roots.

We will work mainly with the following parabolic subgroup of $GL_{m+n}$. Let $P_{m, n}$ be the parabolic 
subgroup of  $GL_{m+n}$ which is given by its Levi decomposition  $P_{m+n}= M_{m+n} U_{m+n}$, where 
\begin{align*}
&M_{m, n}(\Bbb Q)=\left\{\left(\begin{matrix}a& 0\\ 0 & d\end{matrix}\right); \ \ a\in GL_m(\Bbb Q), 
\ \ d\in GL_n(\Bbb Q) \right\} \ \ \text{is a a Levi factor, and}\\
& U_{m, n}(\Bbb Q)=\left\{\left(\begin{matrix}I_m& b\\ 0 & I_n\end{matrix}\right); \ \ b\in M_{m\times n}(\Bbb Q)
\right\} \ \ \text{is the unipotent radical.}
\end{align*}

Let $\chi, \mu: \Bbb Q^\times\backslash  \Bbb A^\times\longrightarrow  \Bbb C^\times$ be unitary
 gr\" ossen characters. 
We consider the usual induced representation:
$$
I(s)\overset{def}{=}\Ind_{P_{m, n}(\Bbb A)}^{GL_{m+n}(\Bbb A)}\left(|\det |^{s/2}\chi \otimes |\det |^{-s/2}\mu\right)
$$
on the space of all $C^\infty$ and right $K$--finite functions 
$f: GL_{m+n}(\Bbb A)\longrightarrow \Bbb C$
which satisfy
\begin{align*}
&f\left(\left(\begin{matrix}a& 0\\ 0 & d\end{matrix}\right)
\left(\begin{matrix}I_m& b\\ 0 & I_n\end{matrix}\right) g\right)=
|\det a|^{s/2}\chi(\det a) |\det d|^{-s/2}\mu(\det d)
\delta_{P_{m, n}}^{1/2}\left(\left(\begin{matrix}a& 0\\ 0 & d\end{matrix}\right)\right)f(g), \\
&\text{where} \ \ 
\left(\begin{matrix}a& 0\\ 0 & d\end{matrix}\right)\in M_{m, n}(\Bbb A), \ 
\left(\begin{matrix}I_m& b\\ 0 & I_n\end{matrix}\right) \in U_{m, n}(\Bbb A), \
g\in  GL_{m+n}(\Bbb A).
\end{align*}

We  construct holomorphic sections $f_s\in I(s)$ using the compact
picture produced by the above choice of a maximal compact subgroup
$K$.  When $Re(s) > \frac{m+n}{2}$, the degenerate Eisenstein series 
\begin{equation}\label{de-1}
E(f_s)(g)\overset{def}{=}\sum_{\gamma \in P_{m, n}(\Bbb Q)\backslash GL_{m+n}(\Bbb Q)}
f_s(\gamma \cdot g)
\end{equation}
 converges absolutely and uniformly  in $(s, g)$ on compact sets. This
 is proved by the restriction to $GL_{m+n}(\mathbb R)$ and then applying
 Godement's theorem as in (\cite{Borel1966}, 11.1 Lemma). In particular, it
 has no poles for $Re(s) > \frac{m+n}{2}$. 

It continues 
to a function which is meromorphic in $s$. Outside of poles, it 
is an automorphic form. As usual and more convenient for computations,
we write $E(s,f)$ instead of $E(f_s)$; in this notation $s$ signals
that $f\in I(s)$.

We say that  $s_0 \in \mathbb C$ is a pole of the degenerate Eisenstein
series $E(s, \cdot)$ if there exists $f\in I(s)$ such that $E(s, f)$
has a pole at $s=s_0$ (for some choice of $g\in GL_{m+n}(\Bbb A)$). If $s_0 \in \mathbb C$, then there exists an
integer $l=l_{s_0}\ge 0$ such that $(s-s_0)^lE(s, f)$ is holomorphic
at $s=s_0$ for all $f\in I(s)$, and for one of them the resulting
function is non--zero. The reader may observe that we allow here $l=0$
which means that no pole occurs at $s_0$. In any case, the map 
\begin{multline}\label{de-100}
\begin{CD}
\Ind_{P_{m, n}(\Bbb A)}^{GL_{m+n}(\Bbb A)}\left(|\det |^{s_0/2}\chi
  \otimes |\det |^{-s_0/2}\mu\right)\\  @>f\mapsto  (s-s_0)^lE(s, f)>>
\cal A\left(GL_{m+n}(\Bbb Q)\backslash GL_{m+n}(\Bbb A)\right)_{\chi^m
  \mu^n|\ |^{\frac{s_0}{2}(m-n)}}
\end{CD}
\end{multline}
is an intertwining operator for the action of 
$\left(\mathfrak{gl}(m+n), K_\infty\right)\times
\prod_{p}GL_{m+n}(\Bbb Q_p)$ in the space of automorphic forms
attached to the central character $\chi^m \mu^n |\ |^{\frac{s_0}{2}(m-n)}$.

The poles of the Eisenstein series are the same as the poles of its constant term along 
the minimal parabolic subgroup:
\begin{equation}\label{de-2}
E_{const}(s,f)(g)= \int_{U(\Bbb Q)\setminus U( \Bbb A)} E(s,f)(ug)du.
\end{equation}
The integral in (\ref{de-2}) can be computed by the standard unfolding.
To explain this we introduce some more notation. We let 
\begin{multline}\label{de-4}
\Lambda_s=\chi |\ |^{\frac{s-(m-1)}{2}}\otimes  \chi |\ |^{\frac{s-(m-1)}{2}+1}\otimes
\cdots \otimes \chi |\ |^{\frac{s+(m-1)}{2}}\otimes\\
\otimes \mu |\ |^{\frac{-s-(n-1)}{2}}\otimes  \mu |\ |^{\frac{-s-(n-1)}{2}+1}\otimes
\cdots \otimes \mu |\ |^{\frac{-s+(n-1)}{2}}, \ \ s\in \Bbb C.
\end{multline}
In this way, we obtain a character 
$$
T(\Bbb Q)\setminus T(\Bbb A)\rightarrow \Bbb C^\times.
$$
We extended trivially across $U(\Bbb A)$ and we induce up to the principal series 
$$
\Ind_{T(\Bbb A)U(\Bbb A)}^{GL_{m+n}(\Bbb A)}(\Lambda_s).
$$
We denote by $\overline{U}$ the lower unipotent triangular matrices in $GL_{m+n}$.
Let $w\in W$. Then, the global intertwining operator
$$
M(\Lambda_s, w): \Ind_{T(\Bbb A)U(\Bbb A)}^{GL_{m+n}(\Bbb A)}(\Lambda_s)\longrightarrow 
\Ind_{T(\Bbb A)U(\Bbb A)}^{GL_{m+n}(\Bbb A)}(w(\Lambda_s)),
$$
defined by 
$$
M(\Lambda_s, w)
f =\int_{U(\Bbb A)\cap
w\overline{U}(\Bbb A)w^{-1}}f(\wit{w}^{-1}ug)du
$$
does not depend on the choice of the representative   for
$w$ in $GL_{m+n}(\Bbb Q)$. 

The global intertwining operator
factors into product of local intertwining operators 
$$
M(\Lambda_s, w)f=\otimes_{p\le \infty} A(\Lambda_{s, p} ,\wit{w})f_p.
$$
There is a precise way of normalization of Haar measures used in the 
definition of intertwining operators 
\cite{Sh1}, \cite{Sh2}. Summary can be found in (\cite{Muic1}, Section 2) or 
(\cite{Muic2}, Section 2). The same is with the normalization factor which we explain 
next. The normalization factor for $A(\Lambda_{s, p}, \wit{w})$ is defined by

$$
r(\Lambda_{s, p} , w)=
\prod_{\alpha\in \Sigma_+,  w(\alpha)<0}\frac{ 
L(1, \Lambda_{s, p} \circ\alpha^\vee)\epsilon(1, \Lambda_{s, p}  \circ\alpha^\vee, \psi_v)}{
L(0, \Lambda_{s, p} \circ\alpha^\vee)},$$
where $\alpha^\vee$ denotes the coroot corresponding to the root
$\alpha,$ and $\psi_v$ is an non-degenerate additive character of $\Bbb
Q_p.$
We define the normalized intertwining operator by the following formula:
$$
\cal N(\Lambda_{s, p} , \wit{w})=r(\Lambda_{s, p} , w)
A(\Lambda_{s, p} , \wit{w}).
$$
 Properties of normalized intertwining operators can be found in \cite{Sh1}, \cite{Sh2}.
Again, the summary can be found in (\cite{Muic2}, Theorem 2-5). 

Let us write $\beta$ for the simple root such that $\Delta-\{\beta\}$
determines $P_{m, n}$.  Now, the constant term has the 
following expression (\cite{Muic1}, Lemma 2.1):
$$
E_{const}(s, f)(g)=\sum_{w\in W, \  w(\Delta\setminus\{\beta\})>0} 
M(\Lambda_s, w)f(g)
=\sum_{w\in W, \  w(\Delta\setminus\{\beta\})>0} 
\int_{U(\Bbb A)\cap w\overline{U}(\Bbb A)w^{-1}}f(\wit{w}^{-1}ug)du,
$$
where, by induction in stages, we identify
\begin{equation}\label{de-30000000000000}
f\in \Ind_{P_{m, n}(\Bbb A)}^{GL_{m+n}(\Bbb A)}\left(|\det |^{s/2}\chi \otimes 
|\det |^{-s/2}\mu\right)
\subset 
\Ind_{T(\Bbb A)U(\Bbb A)}^{GL_{m+n}(\Bbb A)}(\Lambda_s).
\end{equation}
This formula can be more refined up to its final form that we use. Let $S$ be the 
finite set of all places including $\infty$ such that for $p\not\in S$ we have that
$\chi_p$, $\mu_p$, $\psi_p$, and $f_p$ are unramified. Then
 we have the following expression:
\begin{equation}\label{de-3}
E_{const}(s, f)(g)= \sum_{w\in W, \  w(\Delta\setminus\{\beta\})>0} r(\Lambda_{s} , w)^{-1}
\left(\otimes_{p\in S} \cal N(\Lambda_{s, p} , \wit{w})f_p\right)\otimes 
\left(\otimes_{p\not\in S} f_{w, p}\right),
\end{equation}
where we let
\begin{equation}\label{de-5}
r(\Lambda_{s} , w)^{-1}\overset{def}{=}\prod_{\alpha\in \Sigma^+,  w(\alpha)<0}\frac{
L(0, \Lambda_{s} \circ\alpha^\vee)}{ 
L(1, \Lambda_{s} \circ\alpha^\vee)\epsilon(1, \Lambda_{s}  \circ\alpha^\vee)},
\end{equation}
and we use a well--known property of normalization
\begin{equation}\label{de-5000000000}
\cal N(\Lambda_{s, p} , \wit{w})f_p =f_{w, p},
\end{equation}
for unramified $f_p$ and $f_{w, p}$.

We end this section with a technical result which will be important
later. It must be well--known but we could not find a convenient reference.

\begin{Lem}
\label{lem:const0} Assume that $f_s\in \I(s)$ is a holomorphic section. 
Let $s_0\in \Bbb  C$ and assume that  $E(s, \cdot )$ has a pole of order $l$ at  $s_0$
(we consider all sections). Put $\varphi=(s-s_0)^l E(s, f_{s})|_{s=s_0}$.
Then, if $\varphi_{const}=0$, then  $\varphi=0$.
\end{Lem}
\begin{proof} Let $Z$ be the center of $GL_{m+n}$. Since
$\varphi_{const}=0$ and $\varphi$
 is supported on the Borel subgroup $B_{m+n}$, we must have that $\varphi$ is a
 cuspidal automorphic form (but it might not have a unitary central
 character). In particular, we can form the inner
 product
$$
I_s\overset{def}{=}\int_{Z(\Bbb A)GL_{m+n}(\Bbb Q)\backslash
  GL_{m+n}(\Bbb A)}\overline{
\varphi(g)} (s-s_0)^l E(s, f_{s})(g) |\det g|^{\frac{n-m}{2(m+n)}(s+\overline{s_0})} dg,
$$
for any $s$ for which $(s-s_0)^l E(s, f_{s})$ is holomorphic. We inserted
the character $|\det g|^{\frac{n-m}{2(m+n)}(s+\overline{s_0})}$ in order to
make the product trivial on $Z(\Bbb A)$. Clearly, it is enough to
show that $I_{s_0}=0$. Since $I_s$ is meromorphic in $s$, it is
enough to show that $I_s=0$ for $Re(s)>0$ sufficiently large. Indeed, 
if $Re(s)>0$ is sufficiently large, then $E(s,f_s)$ is given by a series 
$\sum_{\gamma \in P_{m, n}(\Bbb Q)\backslash GL_{m+n}(\Bbb Q)}
f_s(\gamma \cdot g)$ which converges absolutely and uniformly  in $(s,
g)$ on compact sets. Inserting this expression into the defining integral
for $I_s$ we obtain
$$
I_s=(s-s_0)^l \int_{Z(\Bbb A)P_{m+n}(\Bbb Q)\backslash GL_{m+n}(\Bbb A)}
\overline{\varphi(g)} f_{s}(g)|\det
g|^{\frac{n-m}{2(m+n)}(s+\overline{s_0})}  dg.
$$
By the Iwasawa decomposition and the standard integration formulas, 
we obtain 

\begin{align*}
I_s=& (s-s_0)^l \int_{Z(\Bbb A) M_{m+n}(\Bbb Q)\backslash M_{m+n}(\Bbb A)}\int_K 
\delta_{P_{m, n}}^{-1}(m)
|\det(mk)|^{\frac{n-m}{2(m+n)}(s+\overline{s_0})}f_s(mk) \times \\
& \times \left(\int_{U_{m+n}(\Bbb Q)\backslash U_{m+n}(\Bbb A)}
\overline{\varphi(umk)}du\right) dmdk=0.
\end{align*}
\end{proof}

\section{Construction of Certain Congruence Subgroups}\label{cong-subgrps}

This section is a preparation for the next Section \ref{real-e}. The
reader may want to consult \cite{BJ} for the relation between the
open--compact subgroups of finite adeles $GL_{m+n}(\Bbb A_f)$ and the congruence subgroups of 
$GL_{m+n}(\Bbb R)$.

Let $m, n\ge 1$. $P_{m, n}(\Bbb Q_p) \subset GL_{m+n}(\Bbb Q_p)$
introduced in Section \ref{de} which has a Levi subgroup 
$$
M_{m, n}(\Bbb Q_p)=\left\{\left(\begin{matrix}a& 0\\ 0 & d\end{matrix}\right); \ \ a\in GL_n(\Bbb Q_p), 
\ \ d\in GL_m(\Bbb Q_p)
\right\},
$$
and the unipotent radical
$$
U_{m, n}(\Bbb Q_p)=\left\{\left(\begin{matrix}I_m& b\\ 0 & I_n\end{matrix}\right); \ \ b\in M_{m\times n}(\Bbb Q_p)
\right\}.
$$
The opposite parabolic subgroup $P^-_{m, n}$ has the same Levi factor but the following unipotent radical:
$$
U^-_{m, n}(\Bbb Q_p)=\left\{\left(\begin{matrix}I_m& 0\\ c & I_n\end{matrix}\right); \ \ c\in M_{n\times m}(\Bbb Q_p)
\right\}.
$$

\begin{Lem}\label{cong-subgrps-1}
Let $p$ be a prime number. Let $l\ge 1$ be an integer.  The set $L_p(l)$ of all
matrices 
$$
\left(\begin{matrix}a& b\\ c & d\end{matrix}\right)\in GL_{m+n}(\Bbb Q_p)
$$
such that 
$a \in M_{m\times m}(\Bbb Z_p)$, $b\in M_{m\times n}(\Bbb Z_p)$, $c\in
M_{n\times m}(p^l\Bbb Z_p)$,  $d \in M_{n\times n}(\Bbb Z_p)$,   and
having the determinant in  $\Bbb Z_p^\times$ 
defines an  open--compact subgroup of $GL_{m+n}(\Bbb
Q_p)$ which has the  following Iwahori factorization: 
$$
L_p(l)=\left(L_p(l) \cap U_{m, n}(\Bbb Q_p) \right) 
\left(L_p(l) \cap M_{m, n}(\Bbb Q_p) \right)\left(L_p(l) \cap U^-_{m, n}(\Bbb Q_p) \right).
$$
\end{Lem}
\begin{proof}First, by considering the reduction homomorphism 
$$
M_{(m+n)\times (m+n)}(\Bbb Z_p)\longrightarrow 
M_{(m+n)\times (m+n)}(\Bbb Z_p/ p^l\Bbb Z_p),
$$ 
we find  
$$
\det\left(\begin{matrix}a& b\\ c & d\end{matrix}\right)\equiv \det{a}\cdot \det{d} \ 
(mod \ \ p^l\Bbb Z_p).
$$ 
Since $l\ge 1$, this implies that 
$$
\det{a}\cdot \det{d} \in \Bbb Z_p^\times.
$$
Thus, we obtain
$$
\det{a},  \det{d} \in \Bbb Z_p^\times \implies 
a \in GL_{m}(\Bbb Z_p), \  d \in GL_{n}(\Bbb Z_p).
$$
By definition, $L_p(l)$ is open and compact subset in $GL_n(\Bbb Q_p)$. Also, it follows from the definition that it 
is closed under the multiplication.

Next, we prove the Iwahori factorization. Indeed, since $d \in GL_{n}(\Bbb Z_p)$, the claim follows from
$$
\left(\begin{matrix}a& b\\ c & d\end{matrix}\right)= 
\left(\begin{matrix}a - bd^{-1}c& b \\ 0 &
    d\end{matrix}\right)\left(\begin{matrix}I_m& 0\\ 
d^{-1}c & I_n\end{matrix}\right).
$$

Since $L_p(l)$ is closed under multiplication,  proved Iwahori
factorization shows that  $L_p(l)$  is closed under taking the
inverses. Hence, $L_p(l)$ is a group. \end{proof}

Now, we consider the global situation. Let $N\ge 2$ be a natural number. We decompose 
into prime numbers
$$
N=p_1^{l_1}\cdots  p_u^{l_u}.
$$
We define an  compact subgroup $L(N)$ of $GL_{m+n}(\Bbb A_f)$ in the following way:
$$
L(N)=\prod_{i=1}^uL_{p_i}(l_i)\times \prod_{p\not \in \{p_1, \ldots, p_u\}} GL_{m+n}(\Bbb Z_p).
$$
We extend this definition by letting $L(1)=\prod_p GL_{m+n}(\Bbb Z_p)$.

We consider $GL_{m+n}(\Bbb Q)$ diagonally embedded in $GL_{m+n}(\Bbb A_f)$. Hence, we may consider the 
intersection
$$
GL_{m+n}(\Bbb Q)\cap L(N) 
$$
in $GL_{m+n}(\Bbb A_f)$. We denote this group by $\Gamma_0(N)$. Explicitly, 
$\Gamma_0(N)$  consists of all matrices
$$
\left(\begin{matrix}a& b\\ c & d\end{matrix}\right)
$$
such that $a \in M_{m\times m}(\Bbb Z)$,  $b\in M_{m\times n}(\Bbb
Z)$, $d \in M_{n\times n}(\Bbb Z)$, and $c\in N\cdot M_{n\times m}(\Bbb Z)$,
and having determinant equal to $\pm 1$.

\section{Restriction of Degenerate Eisenstein Series to $GL_{m+n}(\Bbb R)$}\label{real-e}

We consider the induced representation 
$\Ind_{P_{m, n}(\Bbb A)}^{GL_{m+n}(\Bbb A)}\left(|\det |^{s/2}\otimes |\det |^{-s/2}\right)$
introduced in Section \ref{de} but for trivial $\chi$ and $\mu$. We assume that $m\le n$.

In this section we will consider only very special functions
$f=f_\infty \otimes_{p} f_p$ from that induced representation. 
We require the following: 

\begin{itemize}
\item[(i)] $f_\infty \in \Ind_{P_{m, n}(\Bbb R)}^{GL_{m+n}(\Bbb R)}\left(
|\det |^{s/2}_\infty \otimes |\det |^{-s/2}_\infty\right)$
is any $C^\infty$ and $K_\infty$--finite function;

\item[(ii)] Let $S$ be any finite (perhaps empty) set of primes. 
Then, for $p\not\in S$, we let   $f_p 
\in \Ind_{P_{m, n}(\Bbb Q_p)}^{GL_{m+n}(\Bbb Q_p)}\left(|\det |^{s/2}_p\otimes 
|\det |^{-s/2}_p\right)$ is the unique function which is 
right--invariant under $GL_{m+n}(\Bbb Z_p)$ and satisfies $f_p(I_{m+n})=1$.

\item[(iii)] Let us write $S=\{p_1, p_2, \ldots, p_u\}$. Select any integers 
$l_1, l_2, \ldots, l_u\ge 1$. For $p=p_i$ (and $l=l_i$), we define the function
 $f_p 
\in \Ind_{P_{m, n}(\Bbb Q_p)}^{GL_{m+n}(\Bbb Q_p)}\left(|\det |^{s/2}_p \otimes 
|\det |^{-s/2}_p\right)$ is unique function defined by the following requirements:

\begin{align*}
& \supp{(f_p)}\subset P_{m, n}(\Bbb Q_p) L_p(l),\\
& \text{$f_p$ is right $L_p(l)$--invariant, and}\\
&f_p(I_{m+n})=1.
\end{align*}
\end{itemize}
In the computations below, we allow $u=0$ (equivalently, $S=\emptyset$ in (ii)) i.e., we omit (iii).

We write
$$
GL_{m+n}(\Bbb A)=GL_{m+n}(\Bbb R)\times GL_{m+n}(\Bbb A_f), \ \
g=(g_\infty, g_f).
$$
We recall that $GL_n(\Bbb Q)$ is diagonally embedded. We identify $\gamma=
(\gamma, \gamma)$. Using the notation of Section \ref{cong-subgrps},  the
determinant map  $\det : L(N)\rightarrow \prod_{p}\Bbb Z_p^\times$ is
an  epimorphism (use just upper triangular for each $p$ to see this). Hence, the strong approximation
implies that 
$$
GL_{m+n}(\Bbb A)=GL_{m+n}(\Bbb Q) \cdot \left(GL_{m+n}(\Bbb R)\times L(N)\right),
$$
using above identifications. 

\begin{Lem}\label{real-e-1}
Under above assumptions, $E(s,f)(g_\infty, 1)=
\sum_{\gamma \in \Gamma_0(N)\cap P_{m, n}(\Bbb Q)\backslash 
\Gamma_0(N)}   f_\infty(\gamma \cdot g_\infty)$,
$g_\infty \in GL_{m+n}(\Bbb R)$, where the degenerate Eisenstein
 series $E(s,f)$ is introduced in Section  \ref{de}.
\end{Lem}
\begin{proof} First, we consider the case $N\ge 2$. We have the following:

\begin{align*}
E(s,f)(g_\infty, 1)& =\sum_{\gamma \in P_{m, n}(\Bbb Q)\backslash GL_{m+n}(\Bbb Q)}
f(\gamma \cdot (g_\infty, 1))\\
& =\sum_{\gamma \in P_{m, n}(\Bbb Q)\backslash GL_{m+n}(\Bbb Q)}
f(\gamma \cdot g_\infty, \gamma)\\
& =\sum_{\gamma \in P_{m, n}(\Bbb Q)\backslash GL_{m+n}(\Bbb Q)}
f_\infty(\gamma \cdot g_\infty) \prod_{p} f_p(\gamma).
\end{align*}
Next, we compute $f_p(\gamma)$ for $\gamma\in P_{m, n}(\Bbb Q)\backslash GL_{m+n}(\Bbb Q)$.
Let us write
$$
\gamma=\left(\begin{matrix}a& b\\ c & d\end{matrix}\right)
$$
such that $a \in M_{m\times m}(\Bbb Q)$, 
$b\in M_{m\times n}(\Bbb Q)$, $c\in M_{n\times m}(\Bbb Q)$  $d \in M_{n\times n}(\Bbb Q)$.
Then, because of (iii), we see that corresponding term in Eisenstein series is zero 
unless
$$
\gamma\in P_{m, n}(\Bbb Q) U^-_{m, n}(\Bbb Q).
$$
One easily check that this is equivalent with the fact that $d$ is
invertible, and in the class $ P_{m, n}(\Bbb Q)\gamma$ we have a
representative $\left(\begin{matrix}I_m& 0\\ c & I_n\end{matrix}\right)$
which satisfies 
$$
c_{i, j}\equiv 0 \ \left(mod \ \  p^{l_\alpha}_{\alpha}\Bbb Z_{p_\alpha}\right),\ \ \alpha=1, \ldots, u.
$$ 
This implies that we can write
$$
c=N \beta^{-1} \widetilde{c},
$$
where $\beta$ is an integer prime to $N$, and $\widetilde{c}\in
M_{n\times m}(\Bbb Z)$. At this point we use elementary divisor theory
and write $\widetilde{c}$ in the form
$$
\widetilde{c}=\widetilde{c}_1 \widetilde{d}\widetilde{c}_2,
$$
where $\widetilde{c}_1\in GL_n(\Bbb Z)$, $\widetilde{c}_2\in GL_m(\Bbb
Z)$, and $\widetilde{d}\ \in M_{n\times m}(\Bbb Z)$ has diagonal entries of course up to $\min{(m,
n)}$) and all others terms equal to zero. We define a diagonal matrix 
$\widetilde{A}\in GL_m(\Bbb Q)$ as follows. In $\widetilde{A}$ the entry is different
than $1$ if in $\widetilde{d}$ the entry at the
same position is not zero and it has a greatest common factor with $\alpha$ equal to $x$;
then we let $x/\alpha$ in that position in $\widetilde{A}$. 
Next, we define a diagonal matrix  $\widetilde{C}\in GL_n(\Bbb
Q)$. Here we use $m\le n$. At starts as inverses of the corresponding
diagonal entries of $\widetilde{A}$ and ends with $1$'s. We let
$A=\widetilde{A}\widetilde{c}_2$, and $C=\widetilde{c}_1\widetilde{C}$.
Next, we define $\widetilde{B}\in M_{m\times n}(\Bbb Q)$. It has all
entries equal to zero except if at the position $(i,i)$ the matrix
$\widetilde{d}$ has a non--zero entry. We select
$\widetilde{B}_{ii}\in\Bbb Z$ such that $x\alpha^{-1} +N\alpha^{-1}
\widetilde{B}_{ii}\widetilde{d}_{ii}$ is an integer.
Finally, we let $B=\widetilde{A}\widetilde{c}^{-1}_1$. Then, 
$\left(\begin{matrix}A& B\\ 0 & C\end{matrix} \right)\in P_{m, n}(\Bbb Q), 
$
and it has a determinant equal to one. Finally, 
$$
\left(\begin{matrix}A& B\\ 0 & C\end{matrix} \right)
\left(\begin{matrix}I_m& 0\\N \beta^{-1} a & I_n\end{matrix}\right) \in \Gamma_0(N).
$$
Thus, instead of $\gamma$ we can use this expression in order to
compute what we want. It also shows that we can use the summation over 
$$
\Gamma_0(N)\cap P_{m, n}(\Bbb Q)\backslash \Gamma_0(N).
$$
Thus,  we can write 
$$
E(s,f)(g_\infty, 1)
=\sum_{\gamma \in \Gamma_0(N)\cap P_{m, n}(\Bbb Q)\backslash 
\Gamma_0(N)} f_\infty(\gamma \cdot g_\infty) 
\prod_{p} f_p(\gamma).
$$
Applying (ii), this is equal to
$$
\sum_{\gamma \in \Gamma_0(N)\cap P_{m, n}(\Bbb Q)\backslash 
\Gamma_0(N)} f_\infty(\gamma \cdot g_\infty) 
\prod_{i=1}^u f_{p_i}(\gamma).
$$

We compute $f_{p_i}(\gamma)$ using the Iwahori decomposition given by
Lemma \ref{cong-subgrps-1}. Let us write
$$
\gamma=
\left(\begin{matrix}a& b\\ c & d\end{matrix}\right)= 
\left(\begin{matrix}a - bd^{-1}c& b \\ 0 &
    d\end{matrix}\right)\left(\begin{matrix}I_m& 0\\ 
d^{-1}c & I_n\end{matrix}\right)=
\left(\begin{matrix}a(1 - a^{-1}bd^{-1}c)& b \\ 0 & d\end{matrix}\right)\left(\begin{matrix}
I_m& 0\\ d^{-1}c & I_n\end{matrix}\right).$$
Then, by (iii), we find 
$f_{p_i}(\gamma)=1$.

This completes the proof in the case $N\ge 2$. The case $N=1$ is
similar, and it follows from a well--known identity $B_n(\Bbb Q)\setminus
  GL_n(\mathbb Q)=GL_n(\mathbb Z)\cap B_n(\Bbb Q)\setminus  GL_n(\mathbb Z)$,
  where $B_n$ is a Borel subgroup of upper--triangular matrices. (See
  \cite{Borel1966-1}, 1.11 Examples.) Indeed, since  $B_{m+n}\subset  P_{m, n}$, we obtain that 
$P_{m, n}(\Bbb Q)\setminus  GL_{m+n}(\mathbb Q)=GL_{m+n}(\mathbb
  Z)\cap P_{m, n}(\Bbb Q)\setminus  GL_{m+n}(\mathbb Z)$. 
\end{proof}

One can compute the constant term of Eisenstein series given by Lemma
\ref{real-e-1} using (\cite{Muic3}, Lemma 3.3). This can give a nice
explicit formula for the constant term along the Borel subgroup.

According to Lemma \ref{real-e-1}, for a smooth section $f_\infty \in 
\Ind_{P_{m, n}(\Bbb R)}^{GL_{m+n}(\Bbb R)}\left(|\det |^{s/2} \otimes
  |\det |^{-s/2}\right)$, we define the Eisenstein series 
$$
E_\infty(s, f_\infty)(g_\infty)=\sum_{\gamma \in \Gamma_0(N)\cap P_{m, n}(\Bbb Q)\backslash 
\Gamma_0(N)}   f_\infty(\gamma \cdot g_\infty).
$$
Lemma \ref{real-e-1} tells us that outside of poles
\begin{equation}\label{real-e-4}
E_\infty(s, f_\infty)(g_\infty)=E(s,f_\infty \otimes_{p<
  \infty}f_p)(g_\infty, 1),
\end{equation}
where $f_p$ are given by (ii) and (iii) from the beginning of this
section. We say that $E_\infty(s, \cdot)$ has a pole at $s_0$ of order
$l\ge 1$ if $(s-s_0)^lE_\infty(s, f_\infty)|_{s=s_0}$ is holomorphic for all
$f_\infty$, and at least for one of them is not zero. 

\begin{Lem}\label{real-e-3}
The order of pole on $E(s, \cdot \otimes_{p<
  \infty}f_p)$ at $s=s_0$ is $l\ge 1$ if and only if the same is true for 
$E_\infty(s, \cdot)$. Further, if  both holomorphic, $E_\infty(s, \cdot)$  is not 
identically zero  at $s=s_0$ if and only if the same is true for $E(s, \cdot \otimes_{p<
  \infty}f_p)$.
\end{Lem}
\begin{proof} We prove the first claim. Let $f_\infty \in 
\Ind_{P_{m, n}(\Bbb R)}^{GL_{m+n}(\Bbb R)}\left(|\det |^{s/2} \otimes
  |\det |^{-s/2}\right)$. For $Re(s)$ large, we use defining series
to compute (see the proof of previous lemma)
$$
E_\infty(s, f_\infty)(\gamma_0 g_\infty)=E(s, f_\infty \otimes_{p<
  \infty}f_p)(\gamma_0 g_\infty, 1)
= \sum_{\gamma \in P_{m, n}(\Bbb Q)\backslash GL_{m+n}(\Bbb Q)}
f_\infty(\gamma \cdot \gamma_ 0 g_\infty) \prod_{p} f_p(\gamma),
$$
where $\gamma_0\in GL_{m+n}(\Bbb Q)$. By changing $\gamma$ to $\gamma
\gamma_0^{-1}$, we obtain
$$
E_\infty(s, f_\infty)(\gamma_0 g_\infty)
= \sum_{\gamma \in P_{m, n}(\Bbb Q)\backslash GL_{m+n}(\Bbb Q)}
f_\infty(\gamma \cdot g_\infty) \prod_{p} f_p(\gamma\gamma_0^{-1}).
$$
The last expression gives us 
$$
E_\infty(s, f_\infty)(\gamma_0 g_\infty)=
E(s, f_\infty \otimes_{p<
  \infty} f_p)(g_\infty, \gamma_0^{-1}).
$$
So, after analytic continuation, we obtain
$$(s-s_0)^lE_\infty(s, f_\infty)|_{s=s_0}(\gamma_0 g_\infty)=
(s-s_0)^l E(s, f_\infty \otimes_{p<
  \infty} f_p) |_{s=s_0}(g_\infty, \gamma_0^{-1}).
$$
Now, if $E_\infty(s, \cdot)$ has order less than $l$ which is the
order of $E(s, f_\infty \otimes_{p<
  \infty} f_p)$, then the
left--hand side in above expression is zero for all $f_\infty$. But,
then the right--hand side is also zero at all points from 
$GL_{m+n}(\Bbb R)\times GL_{m+n}(\Bbb Q)$, where $GL_{m+n}(\Bbb Q)$
is diagonally embedded in $GL_{m+n}(\Bbb A_f)$. Also, by (ii) and (iii),
$\otimes_{p<  \infty} f_p$ is right--invariant under the open--compact
subgroup $L(N)$ defined in the previous section. Using the strong
approximation, we obtain that 
$$
(s-s_0)^l E(f_\infty \otimes_{p<
  \infty} f_p) |_{s=s_0}=0.$$
This contradiction proves one direction of the first claim of the
lemma; the converse has similar proof. The second
claim has also similar proof.  
\end{proof}

\section{The Main Results About Degenerate Eisenstein series}\label{main-res}

In this section we first state the main results we prove for
adelic Eisenstein series, and then we apply them to study
Eisenstein series from the previous section. We start from the simplest result.

\begin{Thm}\label{main-res-1} Let $s>0$ be a real number. Assume that $m\le n$. Then, the degenerate Eisenstein
  series (\ref{de-1}) is  holomorphic and non--zero for all $s\not \in
  \{\frac{m+n}{2}-\alpha; \ \  \alpha\in \Bbb Z,  \ \ 0\le \alpha<
  \frac{m+n}{2} \}$. Moreover, the map (\ref{de-100}) is an embedding.
\end{Thm}

The next result is more difficult. 

\begin{Thm}\label{main-res-2} Assume that $m\le n$. Then, the degenerate Eisenstein
  series (\ref{de-1}) with $\mu=\chi$ ($=$ to the trivial character
  without loss of the generality), for $s \in
  \{\frac{m+n}{2}-\alpha; \ \  \alpha\in \Bbb Z,  \ \ 0\le \alpha<
  \frac{m+n}{2} \}$,  has at most
  simple pole, and after removing the possible pole the
  image  of the
  map (\ref{de-100}) is the unique spherical irreducible component in 
$\Ind_{P_{m, n}(\Bbb A)}^{GL_{m+n}(\Bbb A)}\left(|\det |^{s/2}
  \otimes |\det |^{-s/2}\right)$. Moreover, the pole occurs precisely
  for $1\le \alpha \le m-1$.
\end{Thm}

We define
\begin{equation}\label{rel-e-2}
w_0=\left(\begin{matrix}0 & I_n\\
I_m & 0\end{matrix}\right),
\end{equation}
where $I_n$ and $I_m$ are identity matrices of the corresponding
sizes. Let $\widetilde{w_0}$ be the representative taken by Shahidi (see
Section \ref{de}) for the Weyl group element represented by $w_0$.

For the applications on degenerate Eisenstein series on
$GL_{m+n}(\mathbb R)$ we need the following observation:

\begin{Cor}\label{cor-main-res-2} Assume that $m\le n$. Consider the degenerate Eisenstein
  series (\ref{de-1}) with $\mu=\chi = 1$. Let $s_0 \in
  \{\frac{m+n}{2}-\alpha; \ \  \alpha\in \Bbb Z,  \ \ 0\le \alpha<
  \frac{m+n}{2}\}$. Then, for $\alpha\ge m$, $E(f_{s_0})$ is non--zero if
 and only  $\cal N(\Lambda_{s_0, p},  \wit{w_0})f_{s_0} \neq 0$, for
  all $p\le \infty$. Next, for 
$0\le \alpha\le m-1$,  $(s-s_0)E(f_{s})$ is
  not equal to zero at $s=s_0$ if and only if $\cal N(\Lambda_{s, p} ,
  \wit{w_0})f_s \neq 0$ at $s=s_0$, for
  all $p\le \infty$. We remark that in both cases the normalized intertwining
  operator is holomorphic (see Lemma \ref{proof-main-res-2-3}).
\end{Cor}
\begin{proof} This is proved in the course of the proof of Theorem
  \ref{main-res-2} (applying the first part of Lemma
  \ref{proof-main-res-2-4}, the expression (\ref{eq:grouping}), Lemma 
\ref{lem:const0}, together with  Lemmas \ref{proof-main-res-2-5},
  \ref{proof-main-res-2-6}, and  \ref{proof-main-res-2-7}); it is important that $\wit{w_0}$ survives the
  cancellation of poles in Lemma \ref{proof-main-res-2-6} by 
Remark \ref{very--important}).
\end{proof} 

\vskip .2in

We prove Theorem \ref{main-res-1} in Section \ref{ct}. 
Theorem \ref{main-res-2} is proved in Sections \ref{ct3000} and \ref{conclusion}.
The following theorem is a complement to Theorem \ref{main-res-2}.

\begin{Thm}\label{main-res-4} Assume that $m\le n$. Then, for $\mu\neq
  \chi$ and  $s \in
  \{\frac{m+n}{2}-\alpha; \ \  \alpha\in \Bbb Z,  \ \ 0\le \alpha<
  \frac{m+n}{2} \}$,   the  map (\ref{de-100}) is an embedding.
The degenerate Eisenstein series (\ref{de-1}) is  holomorphic if
  $\chi_\infty\mu^{-1}_\infty$ is not a
sign character, or if  $\alpha\le m$.
\end{Thm}

In this case we do not compute the order of the pole if $\alpha\ge 1$. The problem is caused with local normalized
intertwining operators on the archimedean component.

\vskip .2in
Now, we  prove the main result about Eisenstein series 
$E_\infty(s, f_\infty)$ given by   (\ref{real-e-4}).

\begin{Thm}\label{main-res-3}  Let $s>0$ be a real number. Assume that
  $m\le n$. Then, we have the following:
\begin{itemize}
\item[(i)] The degenerate Eisenstein
  series $E_\infty(s, f_\infty)$, is holomorphic and non--zero for $s \not \in
  \{\frac{m+n}{2}-\alpha; \ \  \alpha\in \Bbb Z,  \ \ 0\le \alpha<
  \frac{m+n}{2} \}$.

\item[(ii)]  Let $s \in
  \{\frac{m+n}{2}-\alpha; \ \  \alpha\in \Bbb Z,  \ \ m\le \alpha<
  \frac{m+n}{2} \}$. Then, $E_\infty(s, \cdot)$ is holomorphic, and 
$E_\infty(s, f_\infty)\neq 0$   if and only if $\cal N(\Lambda_{s, \infty} ,
  \wit{w_0})f_\infty \neq 0$.

\item[(iii)]  Let $s \in  \{\frac{m+n}{2}-\alpha; \ \  \alpha\in \Bbb Z,  \ \ 0\le \alpha\le m-1 \}$. Then, 
$E_\infty(s, \cdot)$ has at most simple  pole.

\item[(iv)] In the settings of (iii), if 
  $0\le \alpha\le m-1$, then $E_\infty(s, f_\infty)$ has a
  simple pole if and only if $\cal N(\Lambda_{s, \infty} ,
  \wit{w_0})f_\infty \neq 0$. 
\end{itemize}
(We remark that, in any case, if $f_\infty$ is spherical, then $\cal N(\Lambda_{s, \infty} ,
  \wit{w_0})f_\infty \neq 0$  holds by  (\ref{de-5000000000}).)
\end{Thm}
\begin{proof} (i) is a direct consequence of Lemma
  \ref{real-e-3} and Theorem \ref{main-res-1}. Also, 
(iii) is a direct consequence of Lemma  \ref{real-e-3} and Theorem
\ref{main-res-1}. The same is true for holomorphicity of $E_\infty(s,
  \cdot)$ mentioned in (ii). For the other claims we need to develop 
more sophisticated methods.

Let $s \in  \{\frac{m+n}{2}-\alpha; \ \  \alpha\in \Bbb Z,  \ \ 0\le \alpha<
  \frac{m+n}{2} \}$. First, we assume that $\alpha\ge m$.   
By Lemma \ref{real-e-3}, $E_\infty(s, f_\infty)\neq 0$ if and
  only if $E(s, f_\infty \otimes_{p<
  \infty}f_p)$ satisfies the same.  By Corollary
  \ref{cor-main-res-2}, this is equivalent to $\cal N(\Lambda_{s} ,
  \wit{w_0})\left(f_\infty \otimes_{p<
  \infty}f_p\right)  \neq 0$. Equivalently, using (ii) and (iii) of
  Section \ref{real-e} and (\ref{de-5000000000}), we can write 
$$
\cal N(\Lambda_{s, \infty},   \wit{w_0})f_\infty \otimes 
\cal N(\Lambda_{s, p_1},  \wit{w_0})f_{p_1}\otimes \cdots 
\otimes \cal N(\Lambda_{s, p_u},  \wit{w_0})f_{p_u} \otimes_{p\not\in
  \{p_1, \ldots, p_u\}} f_{w_0, p}\neq 0.
$$
In Lemma
\ref{proof-main-res-2-1}, we prove that the induced 
representations $\Ind_{P_{m, n}(\Bbb Q_p)}^{GL_{m+n}(\Bbb Q_p)}\left(|\det |_p^{s/2} \otimes
  |\det |_p^{-s/2}\right)$ are irreducible for all $p\le \infty$. In
this case, the conditions from  Theorem \ref{main-res-3} always
hold. In particular, $N(\Lambda_{s, p},  \wit{w_0})f_{p}\neq 0$ for
all $p \in \{p_1, \ldots, p_u\}$.  This proves (ii). 

For the proof of (iv) we use similar criterion based on Corollary
 \ref{cor-main-res-2}.  Let us recall that by (\ref{ct-4}), $r(\Lambda_{s, p}, w_0)$ is given by 
\begin{equation}\label{x-main-res-1}
\prod_{i=1}^m \frac{L(s+ \frac{n-m}{2}+i, 1)}
{L(s- \frac{n+m}{2}+i, 1)},
\end{equation}
up to a monomial in $p^{-s}$, and, for $s>0$ large enough, we have  (see Section \ref{de})
\begin{equation}\label{x-main-res-2}
\cal N(\Lambda_{s, p},  \wit{w_0})f_{p}(g)= r(\Lambda_{s, p}, w_0)
\int_{U_{n, m}(\Bbb
  Q_p)}f_p(\wit{w_0}^{-1}ug)du, \ \ g\in GL_{m+n}(\Bbb Q_p).
\end{equation}
Now, we consider the case  $\alpha=0$. We show that 
$$
\cal N(\Lambda_{s, p},  \wit{w_0})f_{p}(\wit{w_0})\neq 0.
$$
This will imply the claim. First, by abusing the notation slightly,
for $s>0$ large enough, using (\ref{x-main-res-2}) with $g=\wit{w_0}$,
(iii) of Section \ref{real-e} implies that the
integral $\int_{U_{n, m}(\Bbb  Q_p)}f_p(\wit{w_0}^{-1}u\wit{w_0})du$ is
a constant $\neq 0$  independent of $s$. We can take equal to one for
simplicity. Then, after analytic continuation, for our $s=(m+n)/2-\alpha=(m+n)/2$, we
obtain  
$$
\cal N(\Lambda_{s, p},  \wit{w_0})f_{p}(\wit{w_0})=r(\Lambda_{s, p},
w_0).
$$
But (\ref{x-main-res-1}) shows that $r(\Lambda_{s, p}, w_0)\neq
0$. This proves (iv) with $\alpha=0$. Now, we consider the more
difficult case $1\le \alpha\le m-1$.  In this case above trick does not work since
$r(\Lambda_{s, p}, w_0)=0$.  Then, again by abusing the notation slightly,
for $s>0$ large enough, we use (\ref{x-main-res-2}) with 
$$
g=\wit{w_0} \left(\begin{matrix}I_m& y\\ 0 &
      I_n\end{matrix}\right), \ \  y\in M_{m\times n}(\Bbb Q_p)
$$
instead of $g=\wit{w_0}$. We have
$$
\int_{U_{n, m}(\Bbb  Q_p)}f_p\left(\wit{w_0}^{-1}u\wit{w_0} \left(\begin{matrix}I_m& y\\ 0 &
      I_n\end{matrix}\right)\right)du=
\int_{M_{n\times m}(\Bbb Q_p)}
f_p\left(  \left(\begin{matrix}I_m& 0\\ x &
      I_n\end{matrix}\right)
 \left(\begin{matrix}I_m& y\\ 0 &
      I_n\end{matrix}\right)\right) dx
$$
Since $\supp{(f_p)}\subset P_{m, n}U^-_{m, n}$ (see (iii) in Section 
\ref{real-e}), this integral takes the following form:
\begin{equation}\label{x-main-res-3}
\int
\left|\det{( I_m -y(I_n+xy)^{-1}x)}\right|_p^{s/2+n/2}
\left|\det{(I_n+xy)}\right|_p^{-s/2-m/2} dx,
\end{equation}
where we integrate over the set of all  $x\in M_{n\times m}(\Bbb Q_p)$
satisfying that $\det{(I_n+xy)}\neq 0$ and $(I_n+xy)^{-1}x \in
M_{n\times m}(p^l\Bbb Z_p)$. 

Now, we complete the proof of (iv) considering the case $m=n$ first. 
Then, $y\in M_{n\times n}(\Bbb Q_p)$. In fact, we may assume that $y\in
GL_{n}(\Bbb Q_p)$. Then, we can transform the first function in
(\ref{x-main-res-3}) to the following form:
\begin{align*}
\det{( I_n -y(I_n+xy)^{-1}x)}&=
\det{( yy^{-1} -y(I_n+xy)^{-1}xyy^{-1})}= \det{( I_n -
  (I_n+xy)^{-1}xy)}\\
&=\det{(I_n+xy)}^{-1} \det{(I_n+xy
  -xy)}=\det{(I_n+xy)}^{-1}.
\end{align*}
Thus, the integral is transformed into a more convenient form
\begin{equation}\label{x-main-res-4}
\int
\left|\det{(I_n+xy)}\right|_p^{-s-n} dx.
\end{equation}
After, the substitution $x\leftrightarrow xy$, and then the
substitution $x\leftrightarrow x+1$, the integral becomes
$$
\left|\det{y}\right|^{-n}_p
\int
\left|\det{x}\right|_p^{-s-n} dx,
$$
where we integrate over the set of all  $x\in GL_{n}(\Bbb Q_p)$, 
$x^{-1} \in I_n + M_{n\times n}(p^l\Bbb Z_p)y$. Ignoring the constants we
look at the family of integrals
$$
\int_{GL_{n}(\Bbb Q_p)}
\left|\det{x}\right|_p^{-s-n} 1_{I_n + M_{n\times n}(p^l\Bbb
  Z_p)y}(x^{-1})d^*x=
\int_{GL_{n}(\Bbb Q_p)}
\left|\det{x}\right|_p^{s} 1_{I_n + M_{n\times n}(p^l\Bbb
  Z_p)y}(x)d^*x
$$
where $y$ ranges over $GL_{n}(\Bbb Q_p)$, where
$d^*x=\left|\det{x}\right|_p^{-n}dx$ is a Haar measure on $GL_{n}(\Bbb
Q_p)$.  Let $z\in GL_{n}(\Bbb Q_p)$, then the integral is equal to 
$$
\int_{GL_{n}(\Bbb Q_p)}
\left|\det{xz}\right|_p^{s} 1_{I_n + M_{n\times n}(p^l\Bbb
  Z_p)y}(xz)d^*x= \left|\det{z}\right|_p^{s}
\int_{GL_{n}(\Bbb Q_p)}
\left|\det{x}\right|_p^{s} 1_{z^{-1} + M_{n\times n}(p^l\Bbb
  Z_p)yz^{-1}}(x)d^*x.
$$
The reader my observe that when $y, z$ range over $GL_{n}(\Bbb Q_p)$,
then the characteristic functions $1_{z^{-1} + M_{n\times n}(p^l\Bbb
  Z_p)yz^{-1}}$ span the space of all Schwartz functions on 
$M_{n\times n}(\Bbb Q_p)$. We conclude that the family of integrals 
we obtained is a family of zeta integrals studied by Jacquet and
Godement  attached to the trivial representation of 
$GL_{n}(\Bbb Q_p)$ \cite{J}. They form a fractional ideal in $\Bbb
C[p^{-s}, p^s]$ whose generator is the inverse of the principal $L$--function
$$
L(s-(n-1)/2, 1_{GL_{n}(\Bbb Q_p)})=\prod_{i=0}^{n-1} L(s- i, 1).
$$

Now, we are ready to conclude the proof. In view of the form of the
normalization factor $r(\Lambda_{s, p}, w_0)$ for $m=n$ (see
(\ref{x-main-res-1}), we conclude that after the analytic continuation
for each $s=(m+n)/2-\alpha=n-\alpha$, $1\le \alpha \le n-1$, we can
find $y\in GL_n(\Bbb Q_p)$ such that 
$$
\cal N(\Lambda_{s, p},  \wit{w_0})f_{p}\left(\wit{w_0} \left(\begin{matrix}I_n& y\\ 0 &
      I_n\end{matrix}\right) \right)\neq 0.
$$
This completes the proof of (iv) in the case $m=n$.

Now, we consider the case $m<n$. In order to deal with the integral (\ref{x-main-res-3}), we write 
$x\in  M_{n\times m}(\Bbb Q_p)$ as column 
$x=\left(\begin{matrix}x' \\ x{''}
\end{matrix}\right)$, where $x'\in M_{m\times m}(\Bbb Q_p), x{''}\in M_{(n-m)\times m}(\Bbb Q_p)$. We try to find $
 y\in M_{m\times n}(\Bbb Q_p)$ in the form 
$y=\left(\begin{matrix}y' &  0 \end{matrix}\right)$, where  $y'\in M_{m\times m}(\Bbb Q_p)$ is invertible as it was $y$ in the case $m=n$ above. 
We insert this into (\ref{x-main-res-3}), and find the following:

$$
\int 
\left|\det{( I_m -y'(I_m+x'y')^{-1}x')}\right|_p^{s/2+n/2}
\left|\det{(I_m+x'y')}\right|_p^{-s/2-m/2} \left(\int dx{''}\right) dx',
$$
where the inner integral is over $x{''}\left( I_m -y'(I_m+x'y')^{-1}x')
\right)\in M_{(n-m)\times m}(p^l\Bbb Z_p)$, and outer integral is over 
 $x'\in M_{m\times m}(\Bbb Q_p)$
satisfying that $\det{(I_m+x'y')}\neq 0$ and $(I_m+x'y')^{-1}x' \in
M_{m\times m}(p^l\Bbb Z_p)$. After an obvious change of variables, the inner integral is equal to 
$$
\left|\det{( I_m -y'(I_m+x'y')^{-1}x')}\right|_p^{m-n},
$$ 
up to a volume of $ M_{(n-m)\times m}(p^l\Bbb Z_p)$ which we ignore. 
Note that the matrix inside this expression is 
non--singular (see the case $m=n$ above; see the computation before (\ref{x-main-res-4})), 
and we obtain 
$$
\int 
\left|\det{( I_m -y'(I_m+x'y')^{-1}x')}\right|_p^{s/2 -n/2+m}
\left|\det{(I_m+x'y')}\right|_p^{-s/2-m/2}  dx',
$$
where we integrate over the same area as before. As in the case $m=n$, this
can be transformed into 
$$
\int 
\left|\det{(I_m+x'y')}\right|_p^{-(s+(m-n)/2) -m}  dx'.
$$
As this point we are at (\ref{x-main-res-4}) with $n$ replaced by $m$, and $s$ replaced by 
$s+(m-n)/2$. As before, we obtain the family of zeta integrals which have a common denominator 
equal to the inverse of 
$$
L(s+(m-n)/2-(m-1)/2, 1_{GL_{m}(\Bbb Q_p)})=\prod_{i=0}^{m-1} L(s+(m-n)/2- i, 1)=
\prod_{i=1}^{m} L(s-(m+n)/2- i, 1).
$$
If we compare this to the expression for the normalization factor given by 
(\ref{x-main-res-1}), we conclude the proof of (iv).
\end{proof}

\section{The proof of Theorem \ref{main-res-1}}\label{ct}

The proof of Theorem \ref{main-res-1} we start by series of
preliminary results. Some of them will be useful later for the proof
of our other results.

We use the notation introduced in Section \ref{de}.
Let $X(T)$ be the group of characters of $T$. It is a free $\bbZ$--module 
$$X(T)\simeq \bbZ \phi_1\ \oplus \ \bbZ \phi_2\ \oplus \cdots \ \oplus \bbZ
\phi_{m+n},$$
where $\phi_i$ is defined by  $\phi_i(\diag(t_1,  t_2, \ldots, t_{m+n}))=
t_i$, $1\le i\le m+n$. The Weyl group $W$ acts on $X(T)$ as follows:
$$
p.\phi_i=\phi_{p(i)}.
$$
We have $\Sigma=\Sigma^+\cup (-\Sigma^+)$, and we have:
$$
\Sigma^+= \{\phi_i-\phi_j ; \ 1\le i< j\le m+n \}.
$$
Let $\alpha_i=\phi_i-\phi_{i+1}$, $1\le i \le m+n-1$.  
Then  $\Delta=\{\alpha_1, \alpha_2, \cdots, \alpha_{m+n-1}\}$.
In this notation, we have $\beta=\alpha_m$, where $\beta$ is defined
in Section \ref{de}.

We observe that $p(\phi_i-\phi_j)\in \Sigma^+$, 
for $1\le i< j\le m+n$ if and only if $p(i)< p(j)$. This observation immediately 
implies that 
\begin{multline}\label{ct-1}
\{w\in W; \ 
w(\Delta\setminus \{\beta\})>0\}=\\
=\{p\in W; \ \ 
p(1)<\cdots < p(m),  \ \ p(m+1)<\cdots <p(m+n)\}. 
\end{multline}

\begin{Lem}\label{ct-2} Assume that $s\ge 0$ is a real number. Let
$\Lambda_s$ be given by (\ref{de-4}). For $w$ from the set given 
by (\ref{ct-1}), we have that $w(\Lambda_s)=\Lambda_s$ implies $w=id$,
or  $s=0$ and $m=n$ in which case there are two choices for $w$.
\end{Lem}
\begin{proof} The requirement $w(\Lambda_s)=\Lambda_s$ is equivalent to 
$w^{-1}(\Lambda_s)=\Lambda_s$. It is easy to see that $w^{-1}$ acts on 
$\Lambda_s$ by permutation of the characters according to $w$ 
in the definition of $\Lambda_s$ given by (\ref{de-4}).

If $w(m+1)=m+1$ or $w(m)=m$, then $w=id$.  If not, then $w(m+1)\le m$.
Thus, $w(m+1)=i+1$ for some $0\le i\le m-1$. This implies 
$\chi |\ |^{\frac{s-(m-1)}{2}+i}= \mu |\ |^{\frac{-s-(n-1)}{2}}$.
Hence, $\chi=\mu$ and $s=(m-n)/2-i$. In particular, $s\le (m-n)/2$.

Similarly, if $w(m)\neq m$, then $w(m)=m+j+1$, for some $0\le j\le
n-1$, then $\chi |\ |^{\frac{s+(m-1)}{2}}= \mu |\ |^{\frac{-s-(n-1)}{2}+j}$.
Hence, $\chi=\mu$ and  $s=-(m+n)/2+j+1$. In particular, 
$s \le -(m+n)/2 +(n-1)+1= (n-m)/2$. 

Thus, combining both inequalities for $s$, we obtain $s\le -|m-n|/2$.
Since $s\ge 0$, we obtain $s=0$ and $m=n$. Now, 
$\Lambda_s$ is of the form 
\begin{multline*}
\chi |\ |^{\frac{-(m-1)}{2}}\otimes  \chi |\ |^{\frac{-(m-1)}{2}+1}\otimes
\cdots \otimes \chi |\ |^{\frac{(m-1)}{2}}\otimes\\
\otimes \chi |\ |^{\frac{-(m-1)}{2}}\otimes  \chi |\ |^{\frac{-(m-1)}{2}+1}\otimes
\cdots \otimes \chi |\ |^{\frac{(m-1)}{2}}.
\end{multline*}
Clearly, there are two such $w$.
\end{proof}

Let $w$ be  from the set given 
by (\ref{ct-1}). We compute the poles of $r(\Lambda_{s} , w)^{-1}$
(see (\ref{de-5})). In view of explicit computation given by Section \ref{ct}, 
 we transform the expression for $r(\Lambda_{s} , w)^{-1}$.

First, we need to determine all $\alpha\in \Sigma^+$,  such that $w(\alpha)<0$.
If we write $\alpha=\phi_i-\phi_j$, $1\le i< j\le m+n$, then we require that
$w(i)> w(j)$. In view of the assumption 
$$
w(1)<\cdots < w(m),  \ \ w(m+1)<\cdots <w(m+n),
$$
we must have $i\le m$ and $m+1\le j$. Using (\ref{de-4}), we have
$$
\Lambda_{s}  \circ\alpha^\vee=\chi |\ |^{\frac{s-(m-1)}{2}+i-1}
\mu^{-1} |\ |^{\frac{s+(n-1)}{2} -j+m+1}=\chi \mu^{-1} |\ |^{s+ \frac{m+n}{2}+i-j}.
$$
Thus, the expression that we need to study is
$$
r(\Lambda_{s} , w)^{-1}=\prod_{\substack{1\le i \le m,\\ m+1\le j \le m+n\\ w(i)> w(j)}}
\frac{L(s+ \frac{m+n}{2}+i-j, \chi \mu^{-1})}{L(s+
  \frac{m+n}{2}+i-j+1, \chi \mu^{-1})}\times (\ast),
$$
where $(\ast)$ is the product of the appropriate $\epsilon$ factors.

Let us fix $i$ and consider only product of $j$'s:
$$
\prod_{\substack{m+1\le j \le m+n\\ w(i)> w(j)}}
\frac{L(s+ \frac{m+n}{2}+i-j, \chi \mu^{-1})}{L(s+ \frac{m+n}{2}+i-j+1, \chi \mu^{-1})}.
$$

Again, because of the assumption $w(m+1)<\cdots <w(m+n)$, we see that the set of all $j$'s which 
contribute 
to the product is of the 
form $m+1, \ldots, j_i$, where $w(j_i)< w(i)$ and $w(j_i+1)> w(i)$. It is possible that 
this set is empty. It is empty precisely when $w(i)=i$ since we must have
$$
w(1)< \cdots <w(i-1)< w(i)=i < w(i+1), \cdots, w(m+n).
$$
Thus, if we have $w(i)>i$ (which is equivalent to $w(i)\neq i$), then
we can write
$$
\frac{\prod_{j=m+1}^{j_i}
L(s+ \frac{m+n}{2}+i-j, \chi \mu^{-1})}{\prod_{j=m+1}^{j_i} 
L(s+ \frac{m+n}{2}+i-(j-1), \chi \mu^{-1})}=
\frac{L(s+ \frac{m+n}{2}+i-j_i, \chi \mu^{-1})}
{L(s+ \frac{m+n}{2}+i-m, \chi \mu^{-1})}.
$$
The whole product is equal to
$$
r(\Lambda_{s} , w)^{-1}=\prod_{w(i)>i} \frac{L(s+ \frac{m+n}{2}+i-j_i, \chi \mu^{-1})}
{L(s+ \frac{n-m}{2}+i, \chi \mu^{-1})}\times (\ast).$$

The set of all $i$ such that $w(i)>i$ is of the form 
$m_w, \cdots, m$ for unique $m_w$.  We observe that 
the following sequence consists of consequent numbers 
\begin{equation}\label{ct-40}
w(i), w(j_i+1), \ldots, w(j_{i+1}), w(i+1).
\end{equation}
The number of elements in the subsequence 
\begin{equation}\label{ct-401}
 w(j_{i}+1), \ldots, w(j_{i+1})
\end{equation}
is clearly the same as the number of elements in the sequence 
 $j_{i}+1, \ldots, j_{i+1}$. Thus, it is equal to $j_{i+1}-j_i$. 
On the other hand, in the sequence  (\ref{ct-40}) there is 
$w(i+1)-w(i)+1$ elements. If we remove from the sequence  (\ref{ct-40})
elements $w(i)$ and $w(i+1)$ we arrive at the sequence (\ref{ct-401}). Thus 
the sequence (\ref{ct-401}) has $w(i+1)-w(i)+1-2=w(i+1)-w(i)-1$ elements. 
This implies the following recursion:
\begin{equation}\label{ct-402}
j_{i+1}=j_i+ w(i+1)-w(i)-1, \ \ m_w\le i \le m-1.
\end{equation}
Similarly, between $w(m_w-1)=m_w-1$ and $w(m_w)$, the consecutive 
elements are $w(m+1), \ldots, w(j_{m_w})$. This implies 
\begin{equation}\label{ct-403}
j_{m_w}=w(m_w)-m_w+m.
\end{equation}
Solving (\ref{ct-402}) and (\ref{ct-403}), we find 
$$
j_{i}=w(i)-i+m.
$$

Thus, the whole product is equal to
\begin{equation}\label{ct-4}
r(\Lambda_{s} , w)^{-1}
=
\prod_{i=m_w}^m \frac{L(s+ \frac{n-m}{2}+2i-w(i), \chi \mu^{-1})}
{L(s+ \frac{n-m}{2}+i, \chi \mu^{-1})}\times (\ast).
\end{equation}

We analyze (\ref{ct-4}) assuming that 
\begin{equation}\label{ct3001}
\text{$s\in \Bbb R$, $s\ge 0$, and $n\ge m$}.
\end{equation}
(From the point of view of the local degenerate principal series this is not serious 
assumption; it can be achieved by 
taking the contragredient and replacing the relative position  of the terms.)

\vskip .2in
In the following lemma we use the some properties of automorphic functions attached to
unitary gr\"ossencharacter. 
Let $\lambda: \Bbb Q^{\times}\setminus \Bbb A^{\times}\longrightarrow \Bbb C^{\times}$ be a
unitary gr\"ossencharacter. Then, $L(s, \lambda)$ is holomorphic and non--vanishing 
for $s>1$. At $s=1$ it is still non--vanishing, but it has a simple pole if and only
if $\lambda$ is trivial. For $0< s< 1$, $L(s, \lambda)$ is holomorphic. At $s=0$, 
it is  non--vanishing, but it has a simple pole if and only
if $\lambda$ is trivial. Finally, for $s<0$,  $L(s, \lambda)$ is holomorphic and 
non--vanishing.

\begin{Lem}\label{ct-6} Assume that (\ref{ct3001}) holds. 
Let $w$ be from the set (\ref{ct-1}). Then, $r(\Lambda_{s} , w)^{-1}$
is holomorphic unless $\chi=\mu$ and there exists an integer
$0\le \alpha\le  \left[\frac{m+n}{2}\right]$ such that 
$$
s= \frac{m+n}{2}-\alpha.
$$
Under these conditions, the order of pole at $s$ is equal to the number of 
$i\in \{m_w, \ldots, m \}$ such that 
$$
n-\alpha-1 \le w(i) - 2i \le n-\alpha,
$$
where $m_w$ is defined to be the least $i\le m$ such that $w(i)>i$. 
Moreover, the set of such $i$'s is not empty only if 
$$
m_w\le \alpha+1.
$$
Finally, the order of pole is 
$$
\le 
\begin{cases}
\min{(m, \alpha+ 1)} -m_w, \ \ \text{if $\alpha=m=n$, and $m_w=1$,}\\
 \min{(m, \alpha+ 1)} -m_w+1, \ \ \text{otherwise.}
\end{cases}
$$
\end{Lem}
\begin{proof} We use (\ref{ct-4}). There are two cases. First, we assume that 
$\chi \neq \mu$. In this case all $L$--functions are holomorphic at real $s$. 
Moreover, by (\ref{ct3001}), we have that 
$$
s+ \frac{n-m}{2}+i\ge i\ge 1.
$$
Thus, the denominator is not vanishing. This proves the lemma in this case.

In remainder of the proof we assume that  $\chi=\mu$. 
The denominator is  non--vanishing, and holomorphic unless 
$s=0$, $m=n$, and $i=1$. Since $i\ge m_w$ in (\ref{ct-4}), we must have $m_w=1$ in that 
case. In the case $s=0$, $m=n$, and $m_w=1$, the denominator 
has a simple pole, and it is holomorphic  and non--vanishing otherwise.

We are interested in the poles of  numerator and the  estimate of their  order. 
The poles of numerator are caused by the poles of its factors.
The factor in numerator given by 
\begin{equation}\label{ct3002}
L(s+ \frac{n-m}{2}+2i-w(i), \triv_{\Bbb Q^\times})
\end{equation}
has a pole if and only if 
$$
s+ \frac{n-m}{2}+2i-w(i)\in \{0, 1\}.
$$

Thus, 
$$
w(i)=s+ \frac{n-m}{2}+ 2i-\epsilon, \ \ \epsilon\in \{0, 1\}.
$$
This imposes conditions on $s$.

First, we have that 
\begin{equation}\label{ct-5}
s+ \frac{n-m}{2}\in \Bbb Z,
\end{equation}
and since $w(i+1), \ldots, w(m)>w(i)$, we must have that
$w(i)$ is not too big, i.e., 
$$
n+i\ge w(i)=s+ \frac{n-m}{2}+2i-\epsilon \iff 0\le s \le 
\frac{n+m}{2} -i +\epsilon
$$
This implies that for $s> \frac{n+m}{2}$ no poles occur. Also, 
(\ref{ct-5}) implies that if $\frac{m+n}{2}-s\not \in \Bbb Z$, then 
no poles occur.

The conclusion of above discussion is that we need to
analyze the order of pole for 
$$
s= \frac{m+n}{2}-\alpha, \ \ \alpha\in\Bbb Z, \ 0\le \alpha \le \left[\frac{m+n}{2}
\right].
$$

We reconsider the factor (\ref{ct3002}) of the numerator. We recall that
$m_w\le i \le m$. Also, we have 
$$
s+ \frac{n-m}{2}+2i-w(i)=n+ 2i -w(i) -\alpha.
$$
Thus, (\ref{ct3002}) has a pole if and only if 

\begin{equation}\label{ct3003}
n-\alpha-1 \le w(i) - 2i \le n-\alpha.
\end{equation}
 Assume that there exists at least one $m_w\le i \le m$ such that this holds. Then, for 
all other $m_w\le i' \le m$ the factor
$$
L(s+ \frac{n-m}{2}+2i'-w(i'), \triv_{\Bbb Q^\times})=
L(n+ 2i' -w(i') -\alpha, \triv_{\Bbb Q^\times})\neq 0
$$
since $n+ 2i' -w(i') -\alpha$ is an integer. 
The conclusion of this discussion is that the order of pole of numerator 
is the number of all $m_w\le i \le m$ such that (\ref{ct3003}) is valid.

The function $i\mapsto w(i)-i$ is increasing for $i\in \{1, \ldots, m\}$. Thus, for 
any $i$ satisfying the condition (\ref{ct3003}),
we must have
$$
n-\alpha-1 \le w(i) -2i\le w(m)-m-i\le (m+n)-m-i=n-i \implies i\le \alpha+1.
$$
This implies that
$$
m_w\le i\le \alpha+1\implies m_w\le \alpha+1.
$$
This implies that the order of pole of numerator is 
$$
\le \min{(m, \alpha+ 1)} -m_w+1.
$$
\end{proof}

\vskip .2in
Now, we complete the proof of  Theorem \ref{main-res-1} in the
following easy steps. Put $I_p(s)=\Ind_{P_{m, n}(\Bbb Q_p)}^{GL_{m+n}(\Bbb Q_p)}\left(|\det |_p^{s/2}
 \chi_p \otimes |\det |_p^{-s/2}\mu_p\right)$.

\begin{Lem}\label{proof-main-res-1} Let $w$ be an element of the set
  given by (\ref{ct-1}).  Let $s\in \mathbb R$
  such that $s-(m+n)/2\not\in \Bbb Z$. Then, the local normalized
  intertwining operator $\cal N(\Lambda_{s, p} , \wit{w})$ is
  holomorphic on $\Ind_{T(\Bbb A)U(\Bbb A)}^{GL_{m+n}(\Bbb
  A)}(\Lambda_{s, p})$, and consequently on 
$I_p(s)$,  for all $p\le \infty$.
\end{Lem}
\begin{proof} The claim is clear if $w$ is identity. Otherwise, 
since $w$ in the set given by (\ref{ct-1}), the
  permutation is a shuffle $1, \ldots, m$ among $m+1, \ldots, m+n$
  without changing the order.  Thus means that $\cal N(\Lambda_{s, p}
  , \wit{w})$ is factorized into the product of certain  rank one
  $GL_2$--operators of the form 
\begin{equation}\label{proof-main-res-2}
\Ind^{GL(2)}\left(\chi_p |\ |_p^{\frac{s-(m-1)}{2}+i}\otimes \mu_p |\
  |_p^{\frac{-s+(n-1)}{2}-j}\right)
\longrightarrow 
\Ind^{GL(2)}\left( \mu_p |\
  |_p^{\frac{-s+(n-1)}{2}-j} \otimes \chi_p |\ |_p^{\frac{s-(m-1)}{2}+i}\right),
\end{equation}
where $0\le i\le m-1$, $0\le j\le n-1$. 

But the induced representation 
$\Ind^{GL(2)}\left(\chi_p |\ |_p^{\frac{s-(m-1)}{2}+i}\otimes \mu_p |\
  |_p^{\frac{-s+(n-1)}{2}-j}\right)$ is irreducible  since 
$s-(m+n)/2\not\in \Bbb Z$. Hence, the rank--one normalized
intertwining operator (\ref{proof-main-res-2}) is holomorphic. This is
a well--known fact, but for reader's convenience we include the sketch
of the proof. Let us call this operator $N_1(s)$. By the normalization
procedure, there is also a normalized operator $N_2(s)$ which goes in
the opposite direction and satisfy $N_2(s) N_1(s)=id$ whenever both
are holomorphic. But if in particular point one of them has a pole,
then we can get rid of the poles getting two non--zero operators $N_1$
and $N_2$ which satisfy $N_2N_1=0$. This is clearly impossible since
the induced representation is irreducible.  
\end{proof}

Going back to the proof of Theorem \ref{main-res-1}, we observe that the Eisenstein
series  given by (\ref{de-1}) is holomorphic for $s > \frac{m+n}{2}$. 
But then its constant term (\ref{de-2}) is also holomorphic in that
region. Then $E_{const}(s,f)-f_s$ is also holomorphic. Next,  
$E_{const}(s,f)\neq 0$ whenever $f\neq 0$ since $f_s$ and
$E_{const}(s,f)-f_s$ are linearly independent because of  Lemma
\ref{ct-2}. Now, we apply Lemma \ref{lem:const0} to complete the proof
of the theorem in this case. 

Let us assume that $s \le  \frac{m+n}{2}$. Then, since by the
assumption of Theorem \ref{main-res-1},  $s\not \in
  \{\frac{m+n}{2}-\alpha; \ \  \alpha\in \Bbb Z,  \ \ 0\le \alpha<
  \frac{m+n}{2} \}$, we see that $s - \frac{m+n}{2}\not\in \mathbb Z$.
Now, Lemmas \ref{ct-6} and \ref{proof-main-res-1} show that
$E_{const}(s,f)$ is holomorphic using the expression (\ref{de-3}). At
this point we can complete the proof of the theorem as before.

\section{Preparation For The Proof Of Theorem \ref{main-res-2}; The
  Circular Result} \label{ct3000}

We start this section with some consequences of Lemma \ref{ct-6}
needed in the proof Theorem \ref{main-res-2} but not in the proof of 
Theorem \ref{main-res-1} completed in the previous section. 

According to Lemma \ref{ct-6}, we make the following definition
(assuming that $\chi=\mu$ and $s=\frac{n+m}{2}-\alpha$ for some integer 
$0\le \alpha\le  \left[\frac{m+n}{2}\right]$).
We denote by 
$W_\alpha$ the set of all elements of (\ref{ct-1}) such that 
$$
w(i)=2i+ n-\alpha-\epsilon_i, \ \ 1\le i \le \min{(m, \alpha+ 1)},
$$
where $\epsilon_i\in\{0, 1\}$. 

We remark that 
$$
w(1)=n+2-\alpha-\epsilon_1\ge 1\iff n\ge \alpha+\epsilon_1-1
$$
which is true since $n\ge m$ (see (\ref{ct3001})). We note that 
 $w(1)=1$ if and only if $\alpha=n$ and $\epsilon_1=1$. This implies $\alpha=m=n$
and $m_w>1$. Also,we have 
$$
w(2)=n+4-\alpha-\epsilon_2\ge 3\iff n\ge \alpha+\epsilon_2-1.
$$
Hence, $w(2)>2$. Thus, $m_w\in \{1, 2\}$.

\begin{Cor}\label{ct-8} Assume that (\ref{ct3001}) holds with 
$s= \frac{m+n}{2}-\alpha$ for some integer $0\le \alpha\le  \left[\frac{m+n}{2}\right]$, and  assume that 
$\chi=\mu$. Then, among all $w$ from the set 
(\ref{ct-1}) the maximal order of pole at $s$ of
$r(\Lambda_{s} , w)^{-1}$
is achieved  if and only if  $w\in W_\alpha$. Moreover, the order of pole is given by
$$
b_\alpha\overset{def}{=}\begin{cases}
\min{(m, \alpha+ 1)}-1 \ \ \text{if $\alpha=m=n$;}\\
\min{(m, \alpha+ 1)} \ \ \text{otherwise}.
\end{cases}
$$
\end{Cor}
\begin{proof} We consider two cases:

First, we assume that the relation $\alpha=m=n$ is not fulfilled. Then, Lemma \ref{ct-6} implies that maximal order
of pole $r(\Lambda_{s} , w)^{-1}$ for $s= \frac{m+n}{2}-\alpha$ is achieved if and only 
if 
$w\in W_\alpha$ since by above remarks for them $m_w=1$. The order is equal to
$\min{(m, \alpha+ 1)}$.

Assume that $\alpha=m=n$. Then, if $w\in W_\alpha$ with $m_w=1$,
the order of pole $r(\Lambda_{s} , w)^{-1}$ for $s= \frac{m+n}{2}-\alpha$ is 
$\min{(m, \alpha+ 1)}-1$. If $w\in W_\alpha$ with $m_w>1$, then $m_w=2$, and 
the order of pole $\min{(m, \alpha+ 1)}-1$. Finally, Lemma \ref{ct-6} implies 
that maximal order
of pole $r(\Lambda_{s} , w)^{-1}$ for $s= \frac{m+n}{2}-\alpha$ is achieved if and only 
if $w\in W_\alpha$.
\end{proof}

\vskip .2in

\begin{Exm} \label{ct3004} Assume that $\alpha=0$. Then, $W_\alpha$ is singleton consisting of the 
permutation $w$ defined by: $w(m+1)=1, \ldots, w(m+n)=n, w(1)=n+1, \ldots, w(m)=m+n$.
Indeed, the definition of $W_\alpha$, forces on  $w$ the following: 
$$
w(1)=n+2-\epsilon_1\in \{n+1, n+2\}.
$$
We remark that 
$$
\epsilon_{1}=1
$$
since $w$ belongs to the set (\ref{ct-1}). 
\end{Exm}

\vskip .2in
In general,  $W_\alpha$ is never empty. It contains a non--empty subset 
$W^0_\alpha$ of elements given by
$$
w(i)=2i+ n-\alpha-1, \ \ 1\le i < \min{(m, \alpha+ 1)},
$$
i.e.,
$$
\epsilon_i=1, \ \ 1\le i < \min{(m, \alpha+ 1)},
$$
satisfying conditions of (\ref{ct-1}). It is easy to check that there exists elements 
of $W^0_\alpha$ with 
$$
\epsilon_{\min{(m, \alpha+ 1)}}=1.
$$
Indeed, firstly, we have
$$
w(1)=n+1-\alpha\ge 1.
$$
Secondly, we must assure that there is enough room to insert
exactly $m-\min{(m, \alpha+ 1)}$ elements 
$$
w\left(\min{(m, \alpha+ 1)+1}\right), \ldots, w(m)
$$
between 
$$
w\left(\min{(m, \alpha+ 1)}\right)=2\min{(m, \alpha+ 1)}+n-\alpha-1
$$
and (including)
$$
m+n.
$$ 
But 
$$
m+n-\left(2\min{(m, \alpha+ 1)}+n-\alpha-1\right)=
m+\alpha+1-2\min{(m, \alpha+ 1)}\ge m-\min{(m, \alpha+ 1)},$$
which proves what we want. 

\vskip .2in 

We see that in above inequality we have equality if and only if 
$$
\min{(m, \alpha+ 1)}=\alpha+1 \iff m\ge \alpha+1.
$$
In that case, the set $W^0_\alpha$ is singleton in view of 
(\ref{ct-1}). If $m< \alpha+1$, then we may select 
$$
\epsilon_{\min{(m, \alpha+ 1)}}=0.
$$

\vskip .2in 
In any case, for $w\in W^0_\alpha$, the following must hold
\begin{equation}\label{ct3005}
\begin{aligned}
&w(m+1)=1, \ldots,  w(m+n-\alpha)= n-\alpha, \\
& w(1)=n-\alpha+1,  w(m+n-\alpha + 1)=n-\alpha+2,\\
& \ldots\\
& w(i)=n-\alpha+ 2i-1, w(m+n-\alpha + i)=n-\alpha+ 2i,\\
& \ldots \\
& w\left(\min{(m, \alpha+ 1)}\right)=2\min{(m, \alpha+ 1)}+n-\alpha-\epsilon_{\min{(m, \alpha+ 1)}}.
\end{aligned}
\end{equation}
where we must require that $1\le i< \min{(m, \alpha+ 1)}$.

In the next two lemmas we extend the result proved in Lemma \ref{ct-2}.

\begin{Lem}\label{ct-200} Let $s\in \mathbb R$. Assume that $w_1\neq w_2$ are from the set given 
by (\ref{ct-1}). Then, we have that $w_1(\Lambda_s)=w_2(\Lambda_s)$ 
implies the following:
\begin{itemize}
\item[(i)]  $\chi=\mu$,
\item[(ii)] there exists a finite (non--empty) set $I\subset \{1, \ldots, m\}$
such that $\{s+\frac{m+n}{2}+ i; \ i\in I \}$ is a subset of 
$\{m+1, \ldots, m+n\}$, and 
\begin{align*}
&w_2(i)=w_1(s+\frac{m+n}{2}+ i); \ i\in I,\\
&w_2(s+\frac{m+n}{2}+ i)=w_1(i); \ i\in I,\\
&w_2=w_1 \ \text{on} \ \{1, \ldots, m+n\}-\{i, \ s+\frac{m+n}{2}+ i; \ i\in I \}.
\end{align*}
\end{itemize}
\end{Lem}
\begin{proof} We have $w^{-1}_1w_2(\Lambda_s)=\Lambda_s$.
The remark from the proof of Lemma \ref{ct-2} shows that in the notation 
given by (\ref{de-4}) the character $w^{-1}_1w_2(\Lambda_s)$ is obtained by 
applying $w^{-1}_2w_1$--permutation to the tensor components of 
$\Lambda_s$. This observation gives the following.

Let $i\in\{1, \ldots, m\}$. Then, the $i$--th component of 
$\Lambda_s$ and $w^{-1}_1w_2(\Lambda_s)$ are the same. Let $j=w^{-1}_2w_1(i)$.
There are two cases. 

Firstly, $j\in\{1, \ldots, m\}$, and 
$
\chi |\ |^{\frac{s-(m-1)}{2}+i-1}=\chi |\ |^{\frac{s-(m-1)}{2}+j-1}$.
This implies $i=j$.

Secondly, $j\in\{m+1, \ldots, m+n\}$, and 
$\chi |\ |^{\frac{s-(m-1)}{2}+i-1}=\mu |\ |^{\frac{-s-(n-1)}{2}+j-m-1}$.
This implies $\chi=\mu$, and $j=s+(m+n)/2+i$.

This completely describes the value of $w^{-1}_2w_1(i)$ for $i\in\{1, \ldots, m\}$.
For  the determination of the value of $w^{-1}_2w_1$ on the set $\{m+1, \ldots, m+n\}$, it 
is more convenient to write $j$ instead of $i$. So, let $j\in \{m+1, \ldots, m+n\}$. 
Repeating above argument, we conclude either $w^{-1}_2w_1(j)=j$ or $w^{-1}_2w_1(j)\in 
\{1, \ldots, m\}$.
In the latter case  $\chi=\mu$ and $j=s+\frac{m+n}{2}+ i$, where we write 
$i=w^{-1}_2w_1(j)$. We readily see that $i=w^{-1}_2w_1(j)$ implies $j=w^{-1}_2w_1(i)$. If not, the the first part 
of the proof shows that $w^{-1}_2w_1(i)=i$ or $w_1(i)=w_2(i)$. On the other hand,  
$i=w^{-1}_2w_1(j)$ implies $w_2(i)=w_1(j)$. So, we get $w_1(i)=w_2(i)=w_1(j)$ which implies $i=j$. This is a 
contradiction. 
\end{proof}

\vskip .2in 
\begin{Lem} \label{ct-11}
Assume that (\ref{ct3001}) holds with 
$s= \frac{m+n}{2}-\alpha$ for some integer $0\le \alpha\le  \left[\frac{m+n}{2}\right]$, and  assume that 
$\chi=\mu$. Then, if $w\in W^0_\alpha$, then we denote by $J_w$ the set of all permutations obtained 
from  $w$ in the following way:  
for each finite (possibly empty) set $I\subset \{1, \ldots, \min{(m, \alpha+1)}-1\}$ we define
a permutation $w_I$ in the following way:
\begin{align*}
&w_I(i)=w(m+n-\alpha+ i); \ i\in I,\\
&w_I(m+n-\alpha+ i)=w(i); \ i\in I,\\
&w_I=w \ \text{on} \ \{1, \ldots, m+n\}-\{i, \ m+n-\alpha+ i; \ i\in I \}.
\end{align*}
Then, we have the following:
\begin{itemize}
\item[(i)]If $w\in W^0_\alpha$, then $w_I\in W_\alpha$;
\item[(ii)]If $w\in W^0_\alpha$, then $w_I(\Lambda_{s})=w(\Lambda_{s})$;
\item[(iii)] $W_\alpha$ is a disjoint union of the sets $J_w$, where  $w\in W^0_\alpha$.
\end{itemize}
\end{Lem}
\begin{proof}   Using (\ref{ct3005}), we obtain 
\begin{align*}
&w_I(i)=w(m+n-\alpha+ i)=2i+n-\alpha; \ i\in I,\\
&w_I(m+n-\alpha+ i)=w(i)=2i+n-\alpha-1; \ i\in I,\\
&w_I=w \ \text{on} \ \{1, \ldots, m+n\}-\{i, \ m+n-\alpha+ i; \ i\in I \}.
\end{align*}
From this and (\ref{ct3005}), (i) is obvious. We remark that (ii) follows 
from Lemma  \ref{ct-200}. 
(iii) is obvious. 
\end{proof}

In the rest of the section, we explain in more detail the structure of the
set $I$ as above, in view of $w_1$ and $w_2$ belonging to the set
given by  (\ref{ct-1}). To that end, for $w$ in the set (\ref{ct-1}),
we call the {\bf{orbit}} of $w$ the set of all $w'$ from (\ref{ct-1})
such that $w(\Lambda_s)=w'(\Lambda_s)$ (of course, the orbit of an
element $w$, as we
saw from Lemma \ref{ct-200}, depends on $s$).

In the next two subsections, we describe the orbits  for other
elements of the set (\ref{ct-1}), not only for the ones belonging to
$W_{\alpha},$ as above. In the first subsection we prove that $I$
defined above is always a specific union of intervals we fully describe. In the second
subsection
we calculate the sum of inverses of normalization factors along
the orbits. The result is that this sum is always either holomorphic or
it has a pole of the first order.

For simplicity, we write $1$ instead  of $\triv_{\Bbb Q^\times}$  for
the trivial gr\" ossen character (a notation introduced in the
previous section).

\subsection {Orbits and poles for a general element $w$ belonging to
  the set   (\ref{ct-1})}
We continue to assume $\chi=\mu,$ $s=\frac{m+n}{2}-\alpha,$ where
$\alpha$ for some integer, $0\le \alpha\le [\frac{m+n}{2}].$  
To each element $w$
of the set   (\ref{ct-1}) we attach a set of
$\{(i_j,k_j):j=1,\ldots,k\}.$ This set describes all the possible sets
$I$ from Lemma \ref{ct-200} for which $w_I$ can be formed as to belong
to the set   (\ref{ct-1}) (and we call the elements in $I$
{\bf{admissible changes}}).  The set $\{(i_j,k_j):j=1,\ldots,k\}$ is
formed  in the following way. Let $i_1$ be the smallest element of
$\{m_w,m_w+1,\ldots, m\}$ which can be the starting point of some set
$I$ from which can make a change from $w$ to another element of
(\ref{ct-1}) in a way described in Lemma \ref{ct-200}.
A number $k_1$ denotes the shortest length of the
interval, say $I_1,$  starting from $i_1$ and consisting of
consecutive integers (i.e., consisting of $i_1,i_1+1, \ldots,
i_1+k_1-1$) which is admissible for a change (cf. Lemma \ref{ct-200}). 
Analogously we define intervals $I_2,\ldots, I_k.$ Let $i_2$ be the 
beginning of the next such interval $(i_2>i_1$). Then, $i_2>i_1+k_1-1,$ in other words, (and more generally)
\begin{Lem} For intervals $I_i$ and $I_j$ as above, we have $I_i\cap I_j=\emptyset,$ if $i\neq j.$
\end{Lem}
\begin{proof}
 The conditions that $I_1$ is an admissible interval for change, in the sense of Lemma \ref{ct-200}, are the following:
\begin{multline}
\label{eq:condition1}
w(i_1-1)<w(m+n-\alpha+i_1)<w(m+n-\alpha+i_1+1)<\cdots 
\\<w(m+n-\alpha+i_1+k_1-1)<w(i_1+k_1),
\end{multline}
\begin{equation}
\label{eq:condition2}
w(m+n-\alpha+i_1-1)<w(i_1)<w(i_1+1)<\cdots <w(i_1+k_1-1)<w(m+n-\alpha+i_1+k_1).
\end{equation}
 If $i_1=1$ or $i_1+k_1=m+1$ or $m+n-\alpha+i_1+k_1=m+n+1,$ we just
 drop the corresponding conditions. 
Note that if $n>\alpha$ we automatically have that $m+n-\alpha+i_1-1\ge m+1.$
If $k_1=1,$ it is obvious that $i_2>i_1+k_1-1.$ Otherwise, the
 condition that $I_1$ is of the minimal length says 
that for each $0\le j\le k_1-2$ we have 
\begin{equation}
\label{eq:fail}
w(m+n-\alpha+i_1+j)>w(i_1+j+1) \text{ or } w(i_1+j)>w(m+n-\alpha+i_1+j+1).
\end{equation}
Since we can start an interval of change from $i_2,$ we have
$w(i_2-1)<w(m+n-\alpha+i_2) $ and 
$w(m+n-\alpha+i_2-1)<w(i_2).$ If we assume that $i_2\le i_1+k-1,$ then
by taking $i_1+j=i_2-1$  we get a  
contradiction with the previous claim.
\end{proof}
From this, we conclude that each set $I,$ admissible in the sense of
Lemma \ref{ct-200}, is the union of some set of intervals
$I_1,\ldots, I_k.$ Assume now that we have picked some
$I_j,\,j=1,\ldots, k$ (for our fixed $w$) and now we 
compare the values of $w$ and $w_{I_j}.$ For the simplicity of
notation, we assume $j=1$; this does not reduce the generality of our
considerations in the next two lemmas.
We either have $w(i_1)<w(m+n-\alpha+i_1)$ or the other way
round. Assume that the first possibility occurs. 
Then, from (\ref{eq:condition1}) and (\ref{eq:condition2}) it follows
that we can write down the following numbers as 
consecutive integers
\begin{equation}\label{possibility1}
\begin{aligned}
& w(i_1),\ w(i_1+1), \ \ldots, \ w(i_1+t_1'),\\
&  w(m+n-\alpha+i_1), \ \ldots, \ w(m+n-\alpha+i_1+t_1''), \\ 
& w(i_1+t_1'+1), \ldots, \ w(i_1+t_2'), \ \ldots, \\
& w(i_1+t_{l-1}'+1), \ \ldots, \ w(i_1+t_l'), \
w(m+n-\alpha+i_1+t_{l-1}''+1), \\
&  \ldots, \  w(m+n-\alpha+i_1+t_l''),
\end{aligned}
\end{equation}
where $0\le t_1'<t_2'<\cdots<t_l'=k_1-1,\;0\le t_1''<t_2''<\cdots<t_l''=k_1-1,$
or 
\begin{equation}
\label{possibility2}
\begin{aligned}
&w(i_1),  \ w(i_1+1), \ \ldots,  w(i_1+t_1'),\\
& w(m+n-\alpha+i_1), \ \ldots, \ w(m+n-\alpha+i_1+t_1''), \\ 
& w(i_1+t_1'+1), \ \ldots, \  w(i_1+t_2'), \ \ldots \\
& w(i_1+t_{l-1}'+1), \ \ldots, \  w(i_1+t_l'), \
w(m+n-\alpha+i_1+t_{l-1}''+1), \\
& \ldots,  w(m+n-\alpha+i_1+t_l''), w(i_1+t_l'+1), \ \ldots, w(i_1+t_{l+1}'),
\end{aligned}
\end{equation}
where, again,  $0\le t_1'<t_2'<\cdots<t_l'<t_{l+1}'=k_1-1,\;0\le t_1''<t_2''<\cdots<t_l''=k_1-1.$

Since the interval $I_1$ is the shortest in the above sense, if we plug
in (\ref{eq:fail}) $j=t_1',$ 
we get $w(m+n-\alpha+i_1+t_1')>w(i_1+t_1'+1)$ or
$w(i_1+t_1')>w(m+n-\alpha+i_1+t_1'+1).$ The second possibility, in 
view of (\ref{possibility1}) and  (\ref{possibility2}) is obviously
impossible, so we get $t_1'\ge t_1''+1,$ i.e., 
$t_1'>t_1''.$ If we continue in the same way, by plugging $t_2',$ and
what we have just obtained, we get $t_2'>t_2''$ 
and so on, until we get $t_{l-1}'>t_{l-1}''.$  On the other hand, if
we assume that we are in the case 
(\ref{possibility2}) we plug $j=t_l'<k_1-1$ and obtain that $t_l'+1\le
t_{l-1}'',$ which is impossible since 
$t_{l-1}''<t_{l-1}'<t_l'.$ We get that the second possibility cannot occur.
We have just  proved the following lemma:
 
\begin{Lem}
\label{lem:t} Assume that $I_1=\{i_1,\ldots,i_1+k_1-1\}$ is
admissible interval in the above sense 
(the shortest one starting in $i_1$). Then, if we assume  that
$w(i_1)<w(m+n-\alpha+i_1),$ the possibility 
(\ref{possibility1}) occurs, with
$t_i''<t_i',\;i=1,2,\ldots, l-1$ and $t_l''=t_l'=k_1-1.$
\end{Lem}
Now we calculate the number of poles of $r(\Lambda_s,w)^{-1}$ and of
$r(\Lambda_s,w_{I_1})^{-1},$ but only the contribution to that expressions
(cf. (\ref{ct-4})) coming from $i\in I_1.$ Note that $i_1\ge m_w$ (unless
$m=n=\alpha$ and $i_1=1,$ cf. (\ref{eq:condition2}); we exclude this
case just now) and that $m_w=m_{w_{I_1}}.$ Also, since we assume that
$\chi=\mu,$ there is no epsilon factors in (\ref{ct-4})-i.e., there is
no $(\ast)$ part.

\begin{Lem}
\label{lem:poles}
Assume that $n>\alpha$ and $t\rightarrow 0$. Then, the product of the factors in the
expression for $r(\Lambda_s,w)^{-1}$ (cf. (\ref{ct-4})) coming from
$i\in I_1$ has a pole of the first order obtained for $i=i_1$; in that
case the expression $L( s+t+\frac{n-m}{2} +2i_1-w(i_1),1)$ becomes
$L(t+1,1).$ The product of the factors in the expression for
$r(\Lambda_s,w_{I_1})^{-1}$ coming from $i\in I_1$ has only one pole
of the first order, 
obtained for $i_1+k_1-1;$ in that case, the expression for 
$L( s+t+\frac{n-m}{2} +2(i_1+k_1-1)-w_I(i_1+k_1-1),1)$ becomes $L(t,1).$
\end{Lem}
\begin{proof} First, we calculate the possible poles for $w,$ (sub)interval, by
(sub)interval. In the first subinterval, we look at $w(i_1+r),\;r\in
[0,t_1'].$ 
We note (cf. the definition of $j_i$ in the third section) that
$j_{i_1+r}=m+n-\alpha+i_1-1;$ we know that 
in $w(i_1+r) $ we have a pole if
$w_{i_1+r}=2(i_1+r)+n-\alpha-\epsilon_{i_1+r,1}.$ Now from expression
for 
$j_{i_1+r}$ (\ref{ct-403}) we get the (only) pole for $r=0$. Then
$\epsilon_{i_1}=1$, and if we write 
$s'=s+t,\;t\to 0,\;s=\frac{m+n}{2}-\alpha,$ then the contribution to the pole in (\ref{ct-4}) is $L(t+1,1).$
Now we look at the possible poles in the subsequent intervals, say, in
the $i+1$-th interval ($1\le i\le  l-1$) (so we examine
$w(i_1+t_i'+1),\ldots, w(i_1+t_{i+1}')$). Here we have
$j_{i_1+r}=m+n-\alpha+i_1+t_i'',$ so that 
$w(i_1+r)-i_1-r+m=m+n-\alpha+i_1+t_i'',\;r\in [t_i'+1,t_{i+1}'].$ We
get possible poles if 
$t_i''+\epsilon_{i_1+r}=r\ge t_i'+1.$ But, according to Lemma \ref{lem:t}, this is impossible ($i\le l-1$).

Second, we calculate possible poles for $w_{I_1}.$  As in the first case, we get that we do not have
any poles, except on the 
last (sub)interval, where we have one pole,  obtained for $i_1+k_1-1;$
then  $\epsilon_{i_1+k_1-1}^I=0,$ 
so the contribution (in poles) is given by the function $L(t,1)$ (where $t\to 0,$ as above).
\end{proof}

\subsection {The sums $\sum_{w'\in [w]}r(\Lambda_s,w')^{-1}$ }
Firstly, we elaborate the case $w\in W_{\alpha}.$  Although the general
case also covers this case, it is methodologically easier to deal with
this case first. So,  we continue to assume $\chi=\mu,$
$s=\frac{m+n}{2}-\alpha,$ where $\alpha$ for some integer, $0\le
\alpha\le [\frac{m+n}{2}].$

For $t$ near zero, we introduce $\gamma(t)=L(t,1)+L(-t,1).$ Using the
Laurent expansion of 
$L(t,1)=\frac{c_1}{t}+c_0+c_1t+c_2t^2+\cdots$ near zero, we get
$\gamma(t)=2\sum_{k=0}^{\infty}c_{2k}t^{2k},$ 
so that $\gamma(t)$ is holomorphic and non-zero (in some neighborhood of zero).
We denote $A(t)=\prod_{i=1}^mL(n-\alpha+t+i,1),$ which is holomorphic
and non-zero function in some neighborhood of 
zero if $n>\alpha$ and $A_1(t)=\frac{A(t)}{L(n-\alpha+t+1,1)}.$ Also
denote $B(t)=\prod_{i=2}^{m-\alpha}L(i+t,1)$ 
(if $m-\alpha\ge 2$).

\begin{Lem} 
\label{lem:norm_factors_maximal}
For $t$ in the sufficiently small neighborhood of zero,  $\sum_{w\in W_{\alpha}}r(\Lambda_{s+t},w)^{-1}$ equals:
\begin{enumerate}
\item $\frac{\gamma(t)^m}{A(t)},$ if  $m<\alpha+1\le n.$
\item  $\frac{B(t)}{A(t)}L(t+1,1)\gamma(t)^{\alpha}$ if  $\alpha+1\le m\le n.$
\item $\frac{\gamma(t)^{m-1}}{A_1(t)}(1+\frac{L(t,1)}{L(-t,1)})$ if $m=n=\alpha.$
\end{enumerate}
\end{Lem}

\begin{proof}
We first cover the case $m< \alpha+1\le n.$ In this case,
$W_{\alpha}^0=\{w_1,w_2\}$ ($w_1$ and $w_2$ just differ in in the
value of $\varepsilon$).  Then, $\sum_{w\in
  W_{\alpha}}r(\Lambda_{s+t},w)^{-1}=\sum_{J_{w_1}}r(\Lambda_{s+t},w)^{-1}+\sum_{J_{w_2}}r(\Lambda_{s+t},w)^{-1}$
and (we have all the needed expressions in  (\ref{ct3005}))
 \[\sum_{J_{w_1}}r(\Lambda_{s+t},w)^{-1}=\frac{1}{A(t)}\sum_{\#I=0}^{m-1}\binom{m-1}{\#I}L(t,1)^{\#I}L(t+1,1)^{m-1-\#
   I}L(t+1,1),\]
and  $\sum_{J_{w_2}}r(\Lambda_{s+t},w)^{-1}=\frac{1}{A(t)}\sum_{\#I=0}^{m-1}\binom{m-1}{\#I}L(t,1)^{\#I}L(t+1,1)^{m-1-\# I}L(t,1),$ so that
\[\sum_{w\in W_{\alpha}}r(\Lambda_{s+t},w)^{-1}=\frac{\gamma(t)^{m}}{A(t)}.\]

\noindent In the case $\alpha+1\le m\le n,$ the set $W_{\alpha}^0$ turns out to be a singleton (by the discussion after Example \ref{ct3004}). So, let $W_{\alpha}^0=\{w_1\}.$ Let $B_{w_1}(t)=\prod_{i=\alpha+2}^mL(n-\alpha+t+2i-w_1(i),1)$ (where $w_1(i)=n+i,\;i=\alpha+2,\ldots,m$). We conclude $B_{w_1}(t)=B(t)=\prod_{i=2}^{m-\alpha}L(i+t,1)$ is holomorphic (and non-zero). We conclude, analogously as in the previous case, that for $\alpha>0$ we have  $\sum_{w\in W_{\alpha}}r(\Lambda_{s+t},w)^{-1}=\frac{B_{w_1}(t)}{A(t)}L(t+1,1)\gamma(t)^{\alpha}.$ If $\alpha=0,$ we get  $\sum_{w\in W_{\alpha}}r(\Lambda_{s+t},w)^{-1}=\frac{L(t+1,1)}{A(t)}\prod_{i=2}^mL(t+i,1).$

\noindent If $\alpha=m=n$ the set $W_{\alpha}^0$ consists of  four elements (as $m_w=1$ or $m_w=2$ and $\varepsilon_m$ equals $0$ or $1$; cf. proof of Lemma \ref{ct-40}). We conclude that for $m\ge 3$, the contribution to  $\sum_{w\in W_{\alpha}}r(\Lambda_{s+t},w)^{-1}$ from $J_{w_1}$ equals $\frac{L(t+1,1)}{A_1(t)}\gamma(t)^{m-2},$ the contribution over $w_2$  equals to  $\frac{L(t,1)}{A_1(t)}\gamma(t)^{m-2},$ over $w_3$  equals to $\frac{L(t,1)}{A_1(t)}\gamma(t)^{m-2}$, and over $w_4$ equals $\frac{L(t,1)^2}{L(t+1,1)A_1(t)}\gamma(t)^{m-2}.$ All together, we get $\frac{\gamma(t)^{m-1}}{A(t)}(1+\frac{L(t,1)}{L(-t,1)}).$
If $m=n=\alpha=1$ we directly get
$r(\Lambda_{s+t},w)^{-1}=\frac{L(t,1)}{L(-t,1)}$ where $w$ is the
unique non-trivial element of the Weyl group, and with the
contribution $1$ from the identity element of the Weyl group, the sum of these two normalization factors is $1+\frac{L(t,1)}{L(-t,1)}.$
\noindent
If  $m=n=\alpha=2$  we have $W_{\alpha}=W_{\alpha}^0=\{w_1,w_2,w_3,w_4\}$ where we use the same notation as in the case $m\ge 3.$ By the direct computation, we get  $\sum_{w\in W_{\alpha}}r(\Lambda_{s+t},w)^{-1}=\frac{\gamma(t)}{L(t+2,1)}(1+\frac{L(t,1)}{L(-t,1)}).$
\end{proof}

Now we move to the general case ($w$ does not necessarily belong to
$W_{\alpha}$). We now examine the contribution of all $w_I$ which
belong to the same orbit as $w.$  The set of all such elements we denote
by $[w].$ We again attach to $w$ a set $\{(i_1,k_1),\ldots, (i_r,k_r)\},$
i.e.,  the set of intervals $\{I_1,\ldots,I_r\}$ as was explained
above. In addition, we choose $w$ to be a specific base point; we
chose $w\in [w]$ in a way that for each $j=1,2,\ldots, r$ we have
$w(i_j)<w(m+n-\alpha+i_j)$ (we had that assumption in Lemma \ref{lem:poles}).

\begin{Lem} 
\label{lem:polestogether}
Assume $n>\alpha.$ Let $w$ be a base-point (explained above). Then,
the sum $\sum_{w'\in [w]}r(\Lambda_{s+t},w')^{-1}$ is a holomorphic function for $t=0$ if $\alpha+1>m;$ it may have a pole of at most first order if $\alpha+1\le m.$
\end{Lem}
\begin{proof}
We denote $A(t)=\prod_{i=m_w}^mL(s+t+\frac{n-m}{2}+i,1)$ the common denominator of all the expressions $r(\Lambda_{s+t},w')^{-1},$ for all $w'\in [w].$ Note that this expression is holomorphic and non-zero. We also introduce 
\[B_1(t)=\prod_{\substack{i\in\{m_w,\ldots, m\}\setminus\\ (\cup_{j=1}^rI_j)}}L(s+t+\frac{n-m}{2}+2i-w(i),1)=\prod_{\substack{i\in\{m_w,\ldots, m\}\setminus\\ (\cup_{j=1}^rI_j)}}L(t+n-\alpha+2i-w(i),1)\]
 (this expression is also common for all $w'\in [w],$ since for those indexes $w'(i)=w(i)$).
Let $a_j(t)=\prod_{i\in I_j}L(s+t+\frac{n-m}{2}+2i-w(i),1), b_j(t)=\prod_{i\in I_j}L(s+t+\frac{n-m}{2}+2i-w(m+n-\alpha+i),1)=\prod_{i\in I_j}L(s+t+\frac{n-m}{2}+2i-w_{I_j}(i),1).$
Then, if $w'\in [w],$ there exists $I$ such that $w'=w_{I},$ and for this  $I$ there exists a subset $C=\{l_1,\ldots,l_p\}\subset \{1,2,\ldots,r\}$ such that $I=\cup_{j=1}^pI_{l_j}.$ This means that $r(\Lambda_{s+t},w')^{-1}=\frac{B_1(t)}{A(t)}\prod_{i\in \{1,2,\ldots,r\}\setminus C}a_i(t)\prod_{j\in C}b_j(t).$ We conclude
\[\sum_{w'\in [w]}r(\Lambda_{s+t},w')^{-1}=\frac{B_1(t)}{A(t)}(a_1(t)+b_1(t))\ldots (a_r(t)+b_r(t)).\]
Note that, according to Lemma \ref{lem:poles}, $a_j(t)=L(t+1,1)a_j'(t),\; b_j(t)=L(t,1)b_j'(t),$ where $a_j'$ and $b_j'$ are holomorphic near $t=0.$ So, we can write $a_j'(t)=a_0+a_1t+a_2t^2+\cdots, \; b_j'(t)=b_0+b_1t+b_2t^2+\cdots,$
$L(t,1)=\frac{c_{-1}}{t}+c_0+c_{-1}t+\cdots, L(t+1,1)=L(-t,1)=-\frac{c_{-1}}{t}+c_0-c_1t+\cdots.$
We then get $a_j(t)+b_j(t)=\frac{c_{-1}(b_0-a_0)}{t}+(c_0(a_0+b_0)+c_1(b_1-a_1))+\cdots$ We, of course, have $a_0=a_j'(0)$ and $b_0=b_j'(0).$ We remind the reader on the above expressions for $a_j(t)$ and $b_j(t).$ Now, we use the induction over $l$ (the length of sequences $t_1',\ldots, t_l'$;-tedious, but straightforward computation using that $t_i'>t_i'', i=1,2,\ldots, l-1$) to prove that $b_0=a_0.$  We are done provided the following Lemma.

\begin{Lem}
\label{lem:B(t)}
The function $B_1(t)$ is holomorphic if $m<\alpha+1,$ and may have a pole of order at most one if $\alpha+1\le m.$
\end{Lem}
\begin{proof} First, we prove that the intervals $I_1,\ldots, I_r$ from the previous
lemma are all connected, i.e., there are no holes between $I_i$ and
$I_{i+1}.$ Assume the opposite, e.g., 
for some $l,$ $i_l+k_l-1<i:=i_l+k_l<i_{l+1}.$  Note that
$w(i-1)=w(i_l+k_l-1)<w(m+n-\alpha+i_l+k_l)=w(m+n-\alpha+i)$ 
and $w(m+n-\alpha+i-1)<w(i),$ since $i_l+k_l-1$ is the end of an interval (of
change). Since from $i$ we cannot start any 
interval, either $w(m+n-\alpha+i+t)>w(i+t+1)$ or
$w(i+t)>w(m+n-\alpha+i+t+1),$ for any $t\in \bbZ_{\ge 0}$ for 
which these expressions  make sense. But for $i+t=i_{l+1}-1$ we get a contradiction.

Now we prove that for $i\in \{m_w,\ldots, i_1-1\}$ the expression $L(t+n-\alpha+2i-w(i),1)$ appearing in  $B_1(t)$ has no pole (for $t=0$). Assume the opposite, and let  $i$ be such a point for which a pole appears. Assume firstly that $w(i)=2i+n-\alpha.$   Then, $j_i=m+n-\alpha+i$ so that
\[w(m+n-\alpha+i)<w(i)<w(m+n-\alpha+i+1),\]
\[w(m+n-\alpha+i)<w(i+1).\]
Since we cannot start any interval from $i,$ we must have
$w(i-1)>w(m+n-\alpha+i).$  We continue, similarly, so if
$w(i-2)<w(m+n-\alpha+i-1),$ we again get contradiction. By induction,
we conclude that $w(i-j)>w(m+n-\alpha+i-j+1),$ as long as this makes
any sense; i.e., until we get to $i-j=m_w.$ But we necessarily have
$w(m_w-1)<w(m+n-\alpha+m_w)$ (or $m_w=1$; in both cases we could start
interval of change from $m_w$). So this cannot happen. We then assume
that for $i<i_1$ we have $w(i)=2i+n-\alpha-1.$ Now we proceed
analogously, i.e., by considering the obstacles for not starting interval of change at $i,$ since $w(m+n+i-\alpha-1)<w(i)<w(m+n-\alpha+i)<w(m+n-\alpha+i+1),$ and so that $w(m+n-\alpha+i)>w(i-1),$ we must have $w(m+n-\alpha+i) >w(i+1).$ Inductively, this would mean that $w(m+n-\alpha+i+t)>w(i+t+1), $ for every $t\in \bbZ_{\ge 0}$ for which these expressions make sense. This means also that $j_{i+t+1}\le m+n-\alpha+i+t-1,$ so that we cannot have a pole for  $w(i+t+1).$ But we do have a pole for $i+t+1=i_1.$

Now assume that $i>i_r+k_r-1$ and that $B_1(t)$ has a pole for that $i.$ If we would assume that $w(i)=2i+n-\alpha,$ this would, as in the case of $i<i_1,$ by descending indices, lead to conclusion  $w(i-j)>w(m+n-\alpha+i-j+1),$ unless we have  a start of some interval at $i-j+1.$ But we do have a start of an interval at $i_r,$ but this would mean that this interval does not end by $i_r+k_r-1,$ but that it lasts until $i.$ A contradiction.

Now we find  the smallest $i>i_r+k_r-1,$ such that we have a pole for $w(i);$ we see that we necessarily have $w(i)=2i+n-\alpha-1.$  By the ascending argument as above, we conclude that  $w(m+n-\alpha+i+t)>w(i+t+1),$ for every $t$ for which expression makes sense; this also means that there are no poles for $i+t+1>i$ (so that $B_1(t)$ has at most one pole for $i>i_r+k_r-1$). Assume now that $\alpha+1>m$ and that we are in this situation (of  $B_1(t)$ having pole). If we plug $t+i=m$ in these expressions, we see that the condition $w(m+n-\alpha+i+t)>w(t+i+1)$ is void; so that we would actually have a possibility of changing intervals
\[w(m+n-\alpha+i-1)<w(i)<\ldots<w(m)<w(m+n-\alpha+m-1)(<w(m+n-\alpha+m+1)),\]
\[w(i-1)<w(m+n-\alpha+i)<\ldots<w(m+n-\alpha+m),\]
and this is a contradiction. So, the only possible pole for $B_1(t)$
occurs for $i>i_r+k_r-1,$ and then only if $\alpha+1\le m;$ this is
what we have seen happening for the Weyl group element $w\in W_{\alpha}$ if $\alpha+1\le m;$ there the interval of change ends with $i=\alpha,$ and the last pole is obtained for $i=\alpha+1.$ 

\end{proof}
\end{proof}
\begin{Rem} \label{very--important}
If $\alpha+1\le m,$ the situations for which $B_1(t),$ and then, consequently,  $\sum_{w'\in [w]}r(\Lambda_{s+t},w')^{-1}$ has a pole of the first order, do occur and not only, as was already noted, for $w$ where the pole is of the maximal order.
E.g., this situation occurs for the element $w_0$
(\ref{rel-e-2}); we have $w_0(i)=n+i.$  Further,  we have
$B_1(t)=\prod_{\substack{i\in\{m_w,\ldots, m\}\setminus\\
    (\cup_{j=1}^rI_j)}}L(-\alpha+t+i,1),$ and we have a pole for
$i=\alpha +1,$ since that $i$ cannot be included in $\cup_{j=1}^rI_j.$
\end{Rem}

We now examine the case of $m=n=\alpha.$ We can check that all the
discussions in this subsection remains the same, 
unless $i_1=1$ in the first interval $I_1$ of possible change for the
element $w.$ Then we do not necessarily have 
$i_1\ge m_w.$ Now assume that $m=n=\alpha,$ and $i_1=1.$ We want to examine the poles of $r(\Lambda_s,w)^{-1}$ and $r(\Lambda_s,w_{I_1})^{-1}.$  Assume  that $w(i_1)=w(1)<w(m+1).$ Then $w(1)=1$ and  Lemma \ref{lem:t} holds. Having in mind (\ref{possibility1}), we have  $m_w=t_1'+2$ and $m_{w_{I_1}}=1.$ We study the analogon of Lemma \ref{lem:poles} in this case. We see that there are no poles in the expression for $r(\Lambda_s,w)^{-1}$ coming from $i\in I_1$ (since there are no poles in the first subinterval, $w(1),\ldots w(1+t_1')$). As for the part of $r(\Lambda_s,w_{I_1})^{-1}$ coming from $i\in I_1,$ which is equal to $\prod_{i=1}^{k}\frac{L(t+2i-w_{I_1}(i))}{L(t+i,1)}$ we again get a pole of the first order in the numerator for $i=k_1$ (a contribution is then $L(t,1)$) but we also have a pole in the denominator $L(t+1,1).$ So, we conclude 
\[r(\Lambda_{s+t},w_{I_1})^{-1}=\prod_{i=1}^{k_1}\frac{L(t+2i-w_{I_1}(i),1)}{L(t+i,1)}\prod_{i=k+1}^m\frac{L(t+2i-w(i),1)}{L(t+i,1)}.\]

Recall that to $w$   we attach a set $\{I_1,\ldots,I_r\}$ of possible
intervals of change, where $I_1=\{1,\ldots,k_1\}.$ We again assume
that $w$ is a base point in its' orbit, i.e.,
$w(i_j)<w(m+i_j),\;j=1,\ldots,r.$ If $I=\cup_{l=1}^sI_{j_s},
\;\{j_1,\ldots,j_l\}\subset\{1,2,\ldots,r\}$ we have two distinct
situations: $I_1\subset I$ and $I_1\not\subset  I.$ If the second possibility occurs, then $m_w=m_{w_I}$ and we can again denote $A(t)=\prod_{i=m_w}^mL(t+i,1),$ the common denominator for all those expressions $r(\Lambda_{s+t},w_I)^{-1}.$ We again introduce $B_1(t)=\prod_{i\in\{m_w,\ldots, m\}\setminus (\cup_{j=1}^rI_j)}L(t+2i-w(i),1).$ We denote by $[w]'$ a part of the orbit $[w]$ comprised of $w_I$ for which $I_1\not\subset I.$ We then have
\[\sum_{w'\in [w]'}r(\Lambda_{s+t},w')^{-1}=\frac{B_1(t)}{A(t)}a_1(t)(a_2(t)+b_2(t))\ldots (a_r(t)+b_r(t)),\]
where $a_j(t)$ and $b_j(t),\;j=2,\ldots r$ are defined as before:  $a_j(t)=\prod_{i\in I_j}L(t+2i-w(i),1), b_j(t)=\prod_{i\in I_j}L(t+2i-w(m+n-\alpha+i),1)$ and $a_1(t)=\prod_{i=m_w}^{k_1}L(t+2i-w(i),1).$ On the other hand, we define $A_1(t)=\prod_{i=1}^mL(t+i,1),$ and $\widetilde{B_1}(t)=\prod_{i\in\{1,\ldots, m\}\setminus (\cup_{j=1}^rI_j)}L(t+2i-w(i),1)$ and then get
\[\sum_{w'\in [w]\setminus [w]'}r(\Lambda_{s+t},w')^{-1}=\frac{\widetilde{B_1(t)}}{A_1(t)}b_1(t)(a_2(t)+b_2(t))\ldots (a_r(t)+b_r(t)),\]
where  $b_1(t)=\prod_{i=1}^{k_1-1}L(t+2i-w(m+i),1)L(t,1).$
 Note that, by the proof of Lemma \ref{lem:B(t)} (the intervals $I_1,\ldots,I_r$ are connected) and the fact that $m_{w}\le k_1+1,$ we have that $B_1(t)=\widetilde{B_1}(t)=\prod_{i=i_r+k_r}^mL(t+2i-w(i),1).$ To conclude, we have
\begin{align*}
\sum_{w'\in [w]}r(\Lambda_{s+t},w')^{-1}=&(a_2(t)+b_2(t))\ldots (a_r(t)+b_r(t))\frac{B_1(t)}{A(t)}\times\\
&(a_1(t)+\frac{\prod_{i=1}^{k_1-1}L(t+2i-w(m+i),1)L(t,1)}{\prod_{i=2}^{m_w-1}L(t+i,1)L(-t,1)}).
\end{align*}
Now, again by induction (as in the proof of Lemma \ref{lem:polestogether}) we prove that $\lim_{t\to 0}\prod_{i=2}^{m_w-1}L(t+i,1)\prod_{i=m_w}^{k_1}L(t+2i-w(i),1)=\lim_{t\to 0}\prod_{i=1}^{k_1-1}L(t+2i-w(m+i),1).$  Since in this case, by Lemma \ref{lem:B(t)}, $B_1(t)$ is holomorphic, we get that
\[\lim_{t\to 0}\sum_{w'\in [w]}r(\Lambda_{s+t},w')^{-1}=0.\]
We have therefore proved the following Lemma.
\begin{Lem}
Assume that $m=n=\alpha,$ and $w$ an element in the Weyl group satisfying (\ref{ct-1}) and such that the first interval of change $I_1$ starts with $i_1=1.$  Then, 
\[\lim_{t\to 0}\sum_{w'\in [w]}r(\Lambda_{s+t},w')^{-1}=0.\]
\end{Lem}

\begin{Rem}
\label{rem1}Note that for any $w$ satisfying  (\ref{ct-1}) with  $m=n=\alpha,$ there is always interval of change  $\{1,2,\ldots,m\}$  and we conclude that $\cup_{j=1}^r I_j=\{1,2,\ldots,m\},$ so that for  $m=n=\alpha$ and  $w$ satisfying  (\ref{ct-1}) we are always in the situation described by the previous lemma.
\end{Rem}

\section{The Conclusion Of  The Proof Of Theorem \ref{main-res-2}}\label{conclusion}

We remind the reader that we assume $\chi=\mu$ is a trivial character,
 and  $s$ is  equal to $\frac{n+m}{2}-\alpha,$ for some  integer $\alpha$ such that $0\le
 \alpha <\frac{m+n}{2}$. 

We work with a  $\Ind_{P_{m, n}(\Bbb A)}^{GL_{m+n}(\Bbb A)}\left(|\det |^{s/2}
  \otimes |\det |^{-s/2}\right)$ for which we find convenient to use
  Zelevinsky notation \cite{Z} for the induced representation:
$\nu^{s/2}\circ \det 1_{GL(m)} \times \nu^{-s/2} \circ \det
  1_{GL(n)}$, where $\nu=|\ |$. Similar notation is used in the local case.

\begin{Lem}\label{proof-main-res-2-0} Let $p\le \infty$. Then, the induced representations 
$\nu^{s/2}\circ \det 1_{GL(m, \Bbb Q_p)} \times \nu^{-s/2} \circ \det
  1_{GL(n, \Bbb Q_p)}$ and 
$\nu^{-s/2}\circ \det 1_{GL(n, \Bbb Q_p)} \times \nu^{s/2} \circ \det
  1_{GL(m, \Bbb Q_p)}$ have the same semi--simplifications in the the
  corresponding Grothendieck group of finite length
  representations. In particular, they share the unique irreducible
  spherical subquotient. 
\end{Lem}
\begin{proof} The representations are induced from associated maximal
  parabolic subgroups and inducing data are associated by acting by
  conjugation with $w_0$ given by (\ref{rel-e-2}). Thus, the general
  theory implies the first claim. The second claim is then obvious.
\end{proof}

\begin{Lem}\label{proof-main-res-2-1} The common irreducible spherical component of 
$\nu^{s/2}\circ \det 1_{GL(m, \Bbb Q_p)} \times \nu^{-s/2} \circ \det
  1_{GL(n, \Bbb Q_p)}$ and $\nu^{-s/2}\circ \det 1_{GL(n, \Bbb Q_p)} \times \nu^{s/2} \circ \det
  1_{GL(m, \Bbb Q_p)}$ is their unique irreducible quotient and
  irreducible subrepresentation, respectively, for all
 $p\le \infty$. For $\alpha=0$, this representation is $\nu^{(m-n)/4}1_{GL_{m+n}(\Bbb Q_p)}$.
Finally, both induced representation are
  irreducible for $m\le \alpha <n$.
\end{Lem} 
\begin{proof} First, we prove the claim if $\alpha\ge 1$. Assume first that $p$ is finite. Then, (\cite{Z},
 Theorem 4.1) implies the irreducibility of $\nu^{s/2}\circ \det
  1_{GL(m, \Bbb Q_p)} \times \nu^{-s/2} \circ \det
  1_{GL(n, \Bbb Q_p)}$ for $m\le \alpha
 <n$. We remind the reader that the assumption in Theorem
 \ref{main-res-2} is $m\le n$ so this case is possible. 
Now, Lemma \ref{proof-main-res-2-0} completes the proof in this case. 

But, if $1\le
 \alpha\le m-1$, then the induced representation 
$\nu^{-s/2}\circ \det 1_{GL(n, \Bbb Q_p)} \times \nu^{s/2} \circ \det
  1_{GL(m, \Bbb Q_p)}$ has its spherical component as a unique
  irreducible subrepresentation (\cite{Z}, Proposition 4.6).
Then,  $\nu^{s/2}\circ \det
  1_{GL(m, \Bbb Q_p)} \times \nu^{-s/2} \circ \det
  1_{GL(n, \Bbb Q_p)}$
 has its spherical  irreducible subquotient as a unique irreducible 
quotient using the result we just proved  and taking the contragredient
 combined with the involution $s$
 introduced in the proof of (\cite{Z}, Theorem 1.9). The involution
  has the following form:
$$
s(g)=s_{m+n}g^{-t}s_{m+n}^{-1}, 
$$
where $s_{m+n}$ is a matrix which has at the position $i, j$ the
element $\delta_{i, m+n+1-j}$ (Kronecker delta), and  $-t$ denotes
transposition and taking an inverse. We recall more details to this
standard procedure below in
the archimedean case.

In the archimedean case, using the same argument as the one 
in (\cite{Z}, Theorem 1.9), one still have the same involution $s$
which shows that, for  every irreducible subrepresentation $\sigma$ of 
$\nu^{-s/2}\circ \det 1_{GL(n, \Bbb Q_p)} \times \nu^{s/2} \circ \det
  1_{GL(m, \Bbb Q_p)}$, $s(\sigma)$ is an irreducible
subrepresentation of $\nu^{s/2}\circ \det 1_{GL(m, \Bbb Q_p)} \times \nu^{-s/2} \circ \det
  1_{GL(n, \Bbb Q_p)}$. Taking the contragredient, we obtain that
  $\widetilde{s(\sigma)}$ is a quotient of 
$\nu^{s/2}\circ \det 1_{GL(m)} \times \nu^{-s/2} \circ \det
  1_{GL(n)}$. All these three induced representations have the
same length. 
Applying (\cite{HL}, Theorems 3.4.2 and 3.4.4) to the first, one sees 
that they are irreducible for  $m\le \alpha <n$, and that
$\nu^{-s/2}\circ \det 1_{GL(n, \Bbb Q_p)} \times \nu^{s/2} \circ \det
1_{GL(n, \Bbb Q_p)}$ has a unique spherical representation, say $\sigma$, as a
unique irreducible subrepresentation for $1\le \alpha\le m-1$. Then, 
$\widetilde{s(\sigma)}$ is a unique irreducible  quotient of 
$\nu^{s/2}\circ \det 1_{GL(m)} \times \nu^{-s/2} \circ \det
  1_{GL(n)}$. But this representation is also spherical, and,
  consequently, by Lemma \ref{proof-main-res-2-0} , isomorphic to 
$\sigma$. By \cite{T}, this is a general fact but we prefer to use
this simple argument in our particular case.

Indeed, let $K_\infty=O(m+n)$ denotes the orthogonal group (a maximal compact subgroup
of $GL_{m+n}(\mathbb R)$. It is immediate that $s(K_\infty)=K_\infty$. In
particular, $\widetilde{s(\sigma)}$ is spherical.

The case $\alpha=0$ is easier. Then $s=(m+n)/2$. Next, 
$\nu^{s/2}\circ \det 1_{GL(m, \Bbb Q_p)}$ is a quotient of the principal
series representation $\nu^{\frac{s}{2}+\frac{m-1}{2}}\times 
\nu^{\frac{s}{2}+\frac{m-1}{2}-1}\times \cdots \times
\nu^{\frac{s}{2}-\frac{m-1}{2}}$
and $\nu^{-s/2}\circ \det 1_{GL(n, \Bbb Q_p)}$ is a quotient of 
$\nu^{-\frac{s}{2}+\frac{n-1}{2}}\times 
\nu^{-\frac{s}{2}+\frac{n-1}{2}-1}\times \cdots \times
\nu^{-\frac{s}{2}-\frac{n-1}{2}}$. Then, 
$\nu^{s/2}\circ \det 1_{GL(m, \Bbb Q_p)} \times \nu^{-s/2} \circ \det
  1_{GL(n, \Bbb Q_p)}$ is a quotient of 
$$
\nu^{\frac{s}{2}+\frac{m-1}{2}}\times 
\nu^{\frac{s}{2}+\frac{m-1}{2}-1}\times \cdots \times
\nu^{\frac{s}{2}-\frac{m-1}{2}}\times 
\nu^{-\frac{s}{2}+\frac{n-1}{2}}\times 
\nu^{-\frac{s}{2}+\frac{n-1}{2}-1}\times \cdots \times
\nu^{-\frac{s}{2}-\frac{n-1}{2}}.$$
But this principal series has $\nu^{(m-n)/4}1_{GL_{m+n}(\Bbb Q_p)}$ as
  its Langlands quotient. 
Therefore, the unique irreducible quotient of $\nu^{s/2}\circ \det 1_{GL(m, \Bbb Q_p)} \times \nu^{-s/2} \circ \det
  1_{GL(n, \Bbb Q_p)}$ is $\nu^{(m-n)/4}1_{GL_{m+n}(\Bbb Q_p)}$.

Similarly, $\nu^{-s/2}\circ \det 1_{GL(n, \Bbb Q_p)} \times \nu^{s/2} \circ \det
  1_{GL(m, \Bbb Q_p)}$ is a subrepresentation of 
$$
\nu^{-\frac{s}{2}-\frac{n-1}{2}}\times 
\nu^{-\frac{s}{2}-\frac{n-1}{2}+1}\times \cdots \times
\nu^{-\frac{s}{2}+\frac{n-1}{2}}\times 
\nu^{\frac{s}{2}-\frac{m-1}{2}}\times 
\nu^{\frac{s}{2}+\frac{m-1}{2}+1}\times \cdots \times
\nu^{\frac{s}{2}+\frac{m-1}{2}}.$$
But this principal series has $\nu^{(m-n)/4}1_{GL_{m+n}(\Bbb Q_p)}$ as
  its Langlands subrepresentation. 
Therefore, the unique irreducible subrepresentation of
 $\nu^{-s/2}\circ \det 1_{GL(n, \Bbb Q_p)} \times \nu^{s/2} \circ \det
  1_{GL(m, \Bbb Q_p)}$ is  $\nu^{(m-n)/4}1_{GL_{m+n}(\Bbb Q_p)}$.
\end{proof}

\begin{Cor}\label{proof-main-res-2-1-c} The intertwining space between 
$\nu^{s/2}\circ \det 1_{GL(m, \Bbb Q_p)} \times \nu^{-s/2} \circ \det
  1_{GL(n, \Bbb Q_p)}$ and $\nu^{-s/2}\circ \det 1_{GL(n, \Bbb Q_p)} \times \nu^{s/2} \circ \det
  1_{GL(m, \Bbb Q_p)}$ is one dimensional. A non---zero intertwining
  operator must take the irreducible quotient of the first onto the
  irreducible subrepresentation of the second.
\end{Cor} 
\begin{proof} This follows immediately from the first claim of Lemma \ref{proof-main-res-2-1}.
\end{proof}

\begin{Lem}\label{proof-main-res-2-2} The local induced  representation $\nu^{s/2}\circ \det 1_{GL(m, \Bbb
   Q_p)} \times \nu^{-s/2} \circ \det
  1_{GL(n, \Bbb Q_p)}$  is generated by its unique unramified vector $f_p$ (see
  Section \ref{de}) for all
 $p\le \infty$. 
\end{Lem} 
\begin{proof} Indeed, the module generated by $f_p$ contains unique
  spherical irreducible subquotient of 
$\nu^{s/2}\circ \det 1_{GL(m, \Bbb Q_p)} \times \nu^{-s/2} \circ \det
  1_{GL(n, \Bbb Q_p)}$ as its irreducible quotient. But applying Lemma
  \ref{proof-main-res-2-1} this module must be all induced
  representation. 
\end{proof}

\begin{Lem}\label{proof-main-res-2-3}  
Let $p\le \infty$. Then, for any $w\in W$ (the Weyl group; see Section \ref{ct}), 
$\cal N_p(\Lambda_{s,p},\tilde{w})$ is holomorphic on  
$\nu^{s/2}\circ \det 1_{GL(m, \Bbb Q_p)} \times \nu^{-s/2} \circ \det
 1_{GL(n, \Bbb Q_p)}$. Furthermore, if $w(\Lambda_s)=w'(\Lambda_s)$, then 
$\cal N_p(\Lambda_{s,p},\tilde{w})$ and $\cal
N_p(\Lambda_{s,p},\tilde{w'})$ agree on 
$\nu^{s/2}\circ \det 1_{GL(m, \Bbb Q_p)} \times \nu^{-s/2} \circ \det
  1_{GL(n, \Bbb Q_p)}$.
\end{Lem}
\begin{proof} The second claim follows from the first, since 
$\nu^{s/2}\circ \det 1_{GL(m, \Bbb Q_p)} \times \nu^{-s/2} \circ \det
  1_{GL(n, \Bbb Q_p)}$ since is generated by its unramified vector (see Lemma
  \ref{proof-main-res-2-2}) and they agree on that unramified vector
  (see (\ref{de-5000000000})). 

The first claim is equally easy. If $\cal N_p(\Lambda_{s,p},\tilde{w})$
would have a pole on $\nu^{s/2}\circ \det 1_{GL(m, \Bbb Q_p)} \times \nu^{-s/2} \circ \det
 1_{GL(n, \Bbb Q_p)}$, then  after removing pole we obtain a non--zero intertwining  operator 
which because of  (\ref{de-5000000000}) has  spherical irreducible
quotient in its kernel. This contradicts Lemma  \ref{proof-main-res-2-2}.
\end{proof}

Let $\widetilde{w_0}$ be the representative taken by Shahidi (see
Section \ref{de}) for the Weyl group element represented by $w_0$ given
by   (\ref{rel-e-2}).

\begin{Lem}\label{proof-main-res-2-4} Let $p\le \infty$.  
Let $f_p \in \nu^{s/2}\circ \det 1_{GL(m, \Bbb Q_p)} \times \nu^{-s/2} \circ \det
 1_{GL(n, \Bbb Q_p)}$. Then, for any $w\in W$, we have the following:
$$
\cal N_p(\Lambda_{s,p},\tilde{w})f_p\neq 0 \iff 
\cal N_p(\Lambda_{s,p},\tilde{w_0})f_p\neq 0.
$$
Furthermore, the image of $\cal N_p(\Lambda_{s,p},\tilde{w})$  is
isomorphic to the unique spherical irreducible quotient of 
$\nu^{s/2}\circ \det 1_{GL(m, \Bbb Q_p)} \times \nu^{-s/2} \circ \det
 1_{GL(n, \Bbb Q_p)}$ (which is well--defined by Lemma
 \ref{proof-main-res-2-1}).
\end{Lem}
\begin{proof} As we proved in Lemma \ref{proof-main-res-2-3}, all 
$\cal N_p(\Lambda_{s,p},\tilde{w})$ are holomorphic on  
$\nu^{s/2}\circ \det 1_{GL(m, \Bbb Q_p)} \times \nu^{-s/2} \circ \det
 1_{GL(n, \Bbb Q_p)}$. The intertwining operator 
$\cal N_p(\Lambda_{s,p},\tilde{w_0})$ is from the space discussed in 
Corollary \ref{proof-main-res-2-1-c}. Therefore, its image is
isomorphic to the unique spherical irreducible quotient of 
$\nu^{s/2}\circ \det 1_{GL(m, \Bbb Q_p)} \times \nu^{-s/2} \circ \det
 1_{GL(n, \Bbb Q_p)}$. Next, we use the following formula that follows 
from the normalization procedure \cite{Sh2}:
$$\cal N_p(\Lambda_{s,p},\tilde{w})=
\cal N_p(\Lambda_{s,p},\widetilde{ww_0^{-1}}) 
\cal N_p(\Lambda_{s,p},\widetilde{w_0}).
$$
As in the previous lemma, we use the property
(\ref{de-5000000000}). The unramified function of the induced representation
$\Ind_{T(\Bbb Q_p)U(\Bbb Q_p)}^{GL_{m+n}(\Bbb Q_p)}(w_0(\Lambda_s))$
belongs to the image of $\cal N_p(\Lambda_{s,p},\tilde{w_0})$ since
it is spherical representation. But then  (\ref{de-5000000000})
implies that $\cal N_p(\Lambda_{s,p},\widetilde{ww_0^{-1}})$ is not
identically equal zero on the image of $\cal
N_p(\Lambda_{s,p},\tilde{w_0})$. It also can not have a pole on that
image because of the same reason. So, above equality of normalized
intertwining operator tells us that all of them have isomorphic
images. Also, the first claim of the lemma follows because of the same 
reason.
\end{proof}

Assume that $\alpha\ge 1$ and $\alpha<n.$ Then, Lemma \ref{proof-main-res-2-3} 
says that we can write down  (\ref{de-3}) as
\begin{equation}
\label{eq:grouping}
E_{const}(s,f)(g)=\sum_w\left (\sum_{w'\in [w]}r(\Lambda_s,w')^{-1}
\right )\otimes_{p\in S}\mathcal{N}(\Lambda_{s,p},\tilde{w})f_p\otimes
(\otimes_{p\notin S}f_{w,p}),
\end{equation}
where the first sum is over some set of representatives of orbits (see
Section \ref{ct3000}) for
the 
set $w\in W, \;w(\Delta\setminus \{\beta\})>0$.

\begin{Lem} \label{proof-main-res-2-5} Assume that $(m+n)/2>\alpha\ge m.$
  Then, the global representation $\nu^{s/2}\circ \det 1_{GL(m, \Bbb
    A)} \times \nu^{-s/2} \circ \det 1_{GL(n, \Bbb A)}$  is irreducible. Moreover, $E(s, \cdot )$ is holomorphic
  and, consequently,  the map (\ref{de-100}) is an embedding.
\end{Lem}
\begin{proof} The irreducibility of the induced representation follows
from Lemma  \ref{proof-main-res-2-1}. We prove that $E(s, \cdot )$ is holomorphic.
By (\ref{eq:grouping}) and Lemma  \ref{proof-main-res-2-3}, it is
enough to show that all  sums $\sum_{w'\in [w]}r(\Lambda_s,w')^{-1}$
are holomorphic. But according to Lemma \ref{lem:polestogether},
$\sum_{w'\in [w]}r(\Lambda_s,w')^{-1}$ is 
holomorphic for all the orbits  and non-zero for some (e.g., for $w$
representing the orbit $W_{\alpha}$ of 
Weyl group elements where the poles of the maximal order of the
normalizing factors are achieved (see Lemma
\ref{lem:norm_factors_maximal}), or for $w=id$). Note that
contributions from the different orbits cannot cancel.
Finally, by (\ref{eq:grouping}), combined with  Lemmas
\ref{proof-main-res-2-4} and \ref{lem:const0}, we see that the map
(\ref{de-100}) has required image.
\end{proof}

\begin{Lem} \label{proof-main-res-2-6}
Assume that $1\le \alpha\le m-1.$ Then, the global representation  
$$
\tau\overset{def}{=}\nu^{-\frac{n-m}{4}}\circ \det 1_{GL(m+n-\alpha,
  \Bbb A)}
\times \nu^{\frac{n-m}{4}}\circ \det 1_{GL(\alpha, \Bbb A)}.
$$
is  irreducible; it is an quotient of the global representation 
$$
\pi\overset{def}{=} \nu^{s/2}\circ \det 1_{GL(m, \Bbb A)} \times \nu^{-s/2} \circ \det
  1_{GL(n, \Bbb A)}.
$$
Moreover, $E(s, \cdot )$ has a simple pole at on $\pi$, and the map
(\ref{de-100}) has image $\tau$. 
\end{Lem}
\begin{proof} The induced representation 
$\tau_p$ is irreducible by (\cite{HL}, Theorems 3.4.2 and 3.4.4) for
$p=\infty$, and by  (\cite{Z}, Proposition 4.6) for $p<\infty$. Hence, 
$\tau=\otimes_{p\le\infty \infty} \tau_p$ is irreducible. Next,  Lemma
\ref{proof-main-res-2-1} implies that $\tau_p$ is a quotient of
$\pi_p$ for all $p$. Thus, $\tau$ is a quotient of $\pi$. 

Again, by (\ref{eq:grouping}) and Lemma  \ref{proof-main-res-2-3}, it is
enough to exhibit poles of the   sums $\sum_{w'\in
  [w]}r(\Lambda_s,w')^{-1}$. 
There exists $w$ from (\ref{ct-1}) such that $\sum_{w'\in
  [w]}r(\Lambda_{s+t},w')^{-1}$ has a simple pole (e. g., $w\in W_{\alpha}$ will  do; see Lemma
\ref{lem:norm_factors_maximal}). All other sums have at most simple
pole (see Lemma \ref{lem:polestogether}). 

Finally, by (\ref{eq:grouping}), combined with  Lemmas
\ref{proof-main-res-2-4} and \ref{lem:const0}, we see that the the map
(\ref{de-100}) has image $\tau$.
\end{proof}

Finally, we consider the following lemma:

\begin{Lem} \label{proof-main-res-2-7}
Assume $\alpha=0.$  Then, the image of the Eisenstein series is a
  realization of the global character  $\nu^{\frac{m-n}{4}}\circ
  \det  1_{GL_{n+m}}(\Bbb A)$.
\end{Lem}
\begin{proof} Since $\alpha=0$ we know, by Example \ref{ct3004}, that
  $W_{\alpha}$ is singleton, say,  
$W_{\alpha}=\{w_0\},$ $r(\Lambda_s,w_0)$ has a simple pole, 
 and for all other $w\in W,\; w(\Delta\setminus\{\beta\})>0$, $r(\Lambda_s,w)$ is holomorphic.
Thus, after cancellation of poles the constant term is of the form:
$$
\left(\lim_{s\to
  \frac{m+n}{2}}(s-\frac{m+n}{2})r(\Lambda_s,w_0)^{-1}\right) \otimes_{p\in
  S}\cal{N}(\Lambda_s,\widetilde{w_0})f_p \otimes \left(\otimes_{p\notin
  S}f_{w_0,p}\right).
$$
Now, we apply Lemmas \ref{proof-main-res-2-1},
\ref{proof-main-res-2-4}, and \ref{lem:const0}.
\end{proof}

Theorem \ref{main-res-2} is consequence of Lemmas
\ref{proof-main-res-2-5}, \ref{proof-main-res-2-6}, and \ref{proof-main-res-2-7}.

\section{Proof of Theorem \ref{main-res-4}}\label{aaaaaaaaaaaaaaaaaaaaa}

The proof of Theorem \ref{main-res-4} is similar to the proof of 
Theorem \ref{main-res-1} (see Section \ref{ct}). We sketch the
argument.  First, the degenerate Eisenstein
  series (\ref{de-1}) has constant term given by (\ref{de-3}). In this
  expression all $r(\Lambda_{s} , w)^{-1}$ are holomorphic at
  indicated points $s$ by Lemma \ref{ct-6}. Now, if all local normalized
  intertwining  operators in the expression (\ref{de-3}) are
  holomorphic at indicated points $s$, we are done by Lemma \ref{ct-2}
  exactly in the same way we concluded the proof of Theorem
  \ref{main-res-1}. If some of them have a pole then, we will show that
  global induced representation is irreducible, and we are again done
  by the argument used in the proof of Lemma \ref{proof-main-res-2-5}.

 Let us fix $p\le \infty$. Let us examine when the local normalized normalized intertwining
operator $\cal N(\Lambda_{s, p} , \wit{w})$ could have poles, where
$w$ in the set given by (\ref{ct-1}). 
If for such $p$ we have $\chi_p=\mu_p$, then as usual we
can assume that it is trivial and the holomorphy of $\cal
N(\Lambda_{s, p} , \wit{w})$  is a part of Lemma
\ref{proof-main-res-2-3}. Let us assume $\chi_p\neq \mu_p$. Then, if
$p< \infty$, then in the decomposition of $\cal N(\Lambda_{s, p} , \wit{w})$ 
operator into rank--one normalized intertwining operators as in the
proof of Lemma \ref{proof-main-res-1} (see (\ref{proof-main-res-1})) 
every operator is holomorphic by the argument used in the proof of 
 Lemma \ref{proof-main-res-1}. Hence, $\cal N(\Lambda_{s, p} ,
 \wit{w})$  itself is holomorphic.

Now, we assume that $p< \infty$. In this case we again look at the
decomposition into rank--one operators (\ref{proof-main-res-1}). 
A typical operator (\ref{proof-main-res-1}) is holomorphic if 
$\frac{s-(m-1)}{2}+i\ge  \frac{-s+(n-1)}{2}-j$ by the general property of
the normalization (for standard representations and parameters in the
closure of positive Weyl chamber). If $\frac{s-(m-1)}{2}+i<
\frac{-s+(n-1)}{2}-j$, the operator is still holomorphic unless the induced
representation is reducible as we explain in the proof of Lemma
\ref{proof-main-res-1}. Since the induced representation is for
$GL_2(\mathbb R)$ the reducibility is easy to establish. Using the
fact that  $\chi_\infty\neq \mu_\infty$, a
necessary  and sufficient condition is that 
\begin{equation}\label{dldldldl}
\begin{aligned}
&\text{$\chi_\infty\mu^{-1}_\infty$ is a
sign character, and}\\ 
&\text{$\frac{-s+(n-1)}{2}-j- \left(\frac{s-(m-1)}{2}+i)\right)>0$   is an
even integer.}
\end{aligned}
\end{equation}

If we have the operator
(\ref{proof-main-res-1}) in the decomposition of $\cal N(\Lambda_{s, \infty} ,
 \wit{w})$, then we also have the operators for the same
$j$ but for all $i\le k\le m-1$ 
$$
\Ind^{GL(2)}\left(\chi_\infty |\ |_\infty^{\frac{s-(m-1)}{2}+k}\otimes \mu_\infty|\
  |_\infty^{\frac{-s+(n-1)}{2}-j}\right)
\longrightarrow 
\Ind^{GL(2)}\left( \mu_\infty |\
  |_\infty^{\frac{-s+(n-1)}{2}-j} \otimes \chi_\infty |\
  |_\infty^{\frac{s-(m-1)}{2}+i}\right).
$$
Since, we are actually interested in $\cal N(\Lambda_{s, \infty} ,
 \wit{w})$ not on the full principal series but on 
\begin{equation}\label{dldldldl-2}
\Ind_{P_{m, n}(\Bbb Q_\infty)}^{GL_{m+n}(\Bbb Q_\infty)}\left(|\det |_\infty^{s/2}
 \chi_\infty \otimes |\det |_\infty^{-s/2}\mu_\infty\right),
\end{equation}
we, as it is usual \cite{MW}, can  replace the collection of such operators with a
single operator
\begin{multline}\label{dldldldl-1}
J\left(\chi_\infty |\ |_\infty^{\frac{s-(m-1)}{2}+i}, \ldots, \chi_\infty |\
|_\infty^{\frac{s+(m-1)}{2}}\right)\times \mu_\infty|\
  |_\infty^{\frac{-s+(n-1)}{2}-j} \\
\longrightarrow \mu_\infty|\
  |_\infty^{\frac{-s+(n-1)}{2}-j} \times
J\left(\chi_\infty |\ |_\infty^{\frac{s-(m-1)}{2}+i}, \ldots, \chi_\infty |\
|_\infty^{\frac{s+(m-1)}{2}}\right),
\end{multline}
where $J(\ldots )$ is a character given as a Langlands
subrepresentation of the principal series 
$\chi_\infty |\ |_\infty^{\frac{s-(m-1)}{2}+i}\times \cdots \times \chi_\infty |\
|_\infty^{\frac{s+(m-1)}{2}}$ of $GL_{m-i}(\mathbb R)$. 

Applying (\ref{dldldldl}), we see that there two possibilities. First,
if $\frac{-s+(n-1)}{2}-j$ is in the segment $\left[\frac{s-(m-1)}{2}+i, \
\frac{s+(m-1)}{2}\right]$. Then, applying the reducibility criterion
of \cite{HL}, we see that the induced representations
figuring in (\ref{dldldldl-1}) are irreducible. Consequently, the
normalized intertwining operator in (\ref{dldldldl-1}) is holomorphic
by an argument used in the proof of Lemma \ref{proof-main-res-1}. 
The other possibility is that 
$$
\frac{-s+(n-1)}{2}-j> \frac{s+(m-1)}{2}.
$$
Then using $s=(m+n)/2-\alpha$, we obtain $\alpha> m+j\ge m$. But in
this case our induced representation (\ref{dldldldl-2}) is irreducible. But since the
same is true for the induced representation 
$\Ind_{P_{m, n}(\Bbb Q_p)}^{GL_{m+n}(\Bbb Q_p)}\left(|\det |_p^{s/2}
 \chi_p \otimes |\det |_p^{-s/2}\mu_p\right)$, $p<\infty$ by \cite{Z}
for $\chi_p\neq \mu_p$, and, by Lemma \ref{proof-main-res-2-1}, 
it follows that the
global induced representation 
$
\Ind_{P_{m, n}(\Bbb A)}^{GL_{m+n}(\Bbb A)}\left(|\det |^{s/2}
 \chi \otimes |\det |^{-s/2}\mu\right)
$
is irreducible. This completes the proof of the theorem.

\end{document}